\documentclass[11pt, reqno]{amsart}

\usepackage{amscd,amsmath}
\usepackage{enumerate}
\usepackage{setspace}
\usepackage{color}
\usepackage{bm}
\usepackage{esint}
\usepackage{float}
 \usepackage{mathrsfs}
\usepackage{amsfonts}
\usepackage{amssymb}
\usepackage{dutchcal}
\usepackage{cancel}
\usepackage{upgreek}
\usepackage{tikz}
\usetikzlibrary{positioning,fadings,through}
\usetikzlibrary{shapes.misc,shapes.geometric,arrows.meta,positioning,shadows.blur,chains,scopes}
\usetikzlibrary{decorations.pathreplacing}

\usetikzlibrary{patterns}
\usepackage[skip=2pt]{caption} % example skip set to 2pt

\usepackage{geometry}
\usepackage[english]{babel}

\usepackage{tabu}

\usetikzlibrary{decorations.pathreplacing}

\usepackage[colorlinks=true, citecolor=blue,urlcolor=red!70!black]{hyperref}
\newtheorem{theorem}{Theorem}[section]
\newtheorem{lemma}[theorem]{Lemma}

\newtheorem{corollary}[theorem]{Corollary}
\newtheorem{definition}[theorem]{Definition}

\newtheorem{remark}[theorem]{Remark}
\newtheorem{proposition}[theorem]{Proposition}

\usepackage{soul}

\newcommand{\eqn}{\begin{eqnarray}}
\newcommand{\een}{\end{eqnarray}}

\newcommand{\pat}{\partial_t}

\definecolor{luh-dark-blue}{rgb}{0.0, 0.313, 0.608}

\newcommand{\bN}{\mathbb{N}}

\numberwithin{equation}{section}

\newcommand{\R}{\mathbb R}

\newcommand{\normaF}[2]{\|#1\|_{\raisebox{-4pt}{\scriptsize $#2$}}}
\newcommand{\normaL}[2]{\left\|#1\right\|_{\raisebox{-2pt}{\scriptsize $#2$}}}
\newcommand{\ignore}[1]{{}}

\usepackage{accents}

\usepackage{enumitem}
\setlist{leftmargin=5.5mm}
\setitemize{itemsep=.8em}

\usepackage{scalerel,stackengine}
\stackMath
\newcommand\reallywidehat[1]{%
\savestack{\tmpbox}{\stretchto{%
  \scaleto{%
    \scalerel*[\widthof{\ensuremath{#1}}]{\kern-.6pt\bigwedge\kern-.6pt}%
    {\rule[-\textheight/2]{1ex}{\textheight}}%WIDTH-LIMITED BIG WEDGE
  }{\textheight}%
}{0.5ex}}%
\stackon[1pt]{#1}{\tmpbox}%
}
\usepackage{mathtools}%

\def\ds{\displaystyle}
\def\eps{\varepsilon}

\def\q{\quad}
\def\qq{\qquad}

\def\intO{\int_\Omega}

\def\lm{\lambda}
\def\de{\delta}
\def\si{\sigma}
\def\ga{\gamma}

\def\pa{\partial}
\def\N{\nabla}
\def\D{\Delta}
\def\vp{\varphi}
\def\al{\alpha}

\def\s{\t{s}}

\newcommand{\pare}[1]{\left( #1 \right)}
\newcommand{\norm}[1]{\left\| #1 \right\|}
\newcommand{\av}[1]{\left| #1 \right|}
\newcommand{\bra}[1]{\left[ #1 \right]}
\newcommand{\set}[1]{\left\{ #1 \right\}}
\newcommand{\sgn}{\textnormal{sgn}}
\renewcommand{\t}[1]{\text{#1}}
\newcommand{\dive}[1]{\t{div}\left( #1 \right)}

\allowdisplaybreaks

\def\sideremark#1{\ifvmode\leavevmode\fi\vadjust{\vbox to0pt{\vss
\hbox to0pt{\hskip\hsize\hskip1em
\vbox{\hsize3cm\tiny\raggedright\pretolerance10000
\noindent #1\hfill}\hss}\vbox to8pt{\vfil}\vss}}}

\usepackage{orcidlink}
\begin{document}

\setlength{\abovedisplayskip}{5pt}
\setlength{\belowdisplayskip}{5pt}

\setlength{\jot}{7pt}

\title[Parabolic equations having a superlinear gradient term]{Quasi-linear parabolic equations having a superlinear gradient term which depends on the solution}

\author[A.  Dall'Aglio]{Andrea Dall'Aglio %\orcidlink{https://orcid.org/0000-0002-4582-798X}
}
\address{Dipartimento di Matematica ``G. Castelnuovo'',
``Sapienza'' Universit{\`a} di Roma, \hfill \break\indent
Piazzale Aldo Moro 2, I-00185 Roma, Italy.}
\email{dallaglio@mat.uniroma1.it\hfill\break\indent
%\orcidlinkc{0000-0002-4582-798X}
}

\author[M. Magliocca]{Martina Magliocca %\orcidlink{https://orcid.org/0000-0002-1188-5348}
}
\address{Departamento de An\'alisis Matem\'atico, Universidad de Sevilla, \hfill \break\indent
Tarfia s/n, 41012 Sevilla, Spain.}
\email{mmagliocca@us.es\hfill\break\indent
%\orcidlinkc{0000-0002-1188-5348}
}

\author[S. Segura de Le\'on]{Sergio Segura de Le\'on %\orcidlink{https://orcid.org/0000-0002-8515-7108}
}
\address{Departament d'An\`alisi Matem\`atica, Universitat de Val\`encia, \hfill \break\indent
Dr. Moliner 50, 46100 Burjassot, Val\`encia, Spain.}
\email{sergio.segura@uv.es\hfill\break\indent
%\orcidlinkc{0000-0002-8515-7108}
}

\date{\today}



\begin{abstract}
In this paper, we study existence and regularity for solutions to parabolic equations having a superlinear lower order term depending on both the solution and its gradient.  Two different situations are analyzed. On the one hand, we assume that the initial datum belongs to an Orlicz space of  exponentially summable functions. On the other, data in an appropriate Lebesgue space satisfying a smallness condition are considered. Our results are coherent with those of previous papers in similar frameworks.
\end{abstract}

\maketitle

\setcounter{tocdepth}{3}

%{\small \tableofcontents}

\section{Introduction}

This paper is devoted to study existence of solutions to parabolic equations having a superlinear lower order term depending on both the solution and its gradient. The prototype is the following problem:
\begin{equation} \label{eq:toy-pb}
\begin{cases}
\begin{aligned}
\pat u-\D u&=\av{u}\av{\N u}^q+f&&\t{in } Q_T=(0,T)\times\Omega,\\
u(0,x)&=u^0(x)&&\t{in }\Omega,\\
u(t,x)&=0&&\t{on }(0,T)\times\partial\Omega,
\end{aligned}
\end{cases}
\end{equation}
with $0<q<2$. Here $\Omega$ is a bounded open set in $\R^N$ ($N\ge 3$) and $T>0$.
By ``superlinear'' we mean that the term in the right hand side of  \eqref{eq:toy-pb}, which depends on $u$ and $\nabla  u$, is homogeneous with total degree greater than one.
Regarding the data $u^0$ and $f$, they belong to suitable functional spaces which will be detailed below.

The assumptions  on the power $q$ imply that the term $|u|\av{\N u}^q$ always grows with superlinear rate, so an essential aspect of any existence result is to determine suitable classes of  functions for the  data.\\
This is a common feature when dealing with superlinear problems, and it is not  related to the fact of having a parabolic problem nor a gradient term. Indeed, even in the elliptic framework (see for instance \cite{GMP, AFM}) or in the parabolic case with superlinear powers of the solution (see for instance \cite{BC,Ppower}) the existence of solutions  is strictly related to the regularity of the data. \\
It is worth noticing that the regularity of initial data affects both the existence and the notion of solutions. As expected, the solution we obtain enjoys the same regularity at each instant $t>0$ as the initial datum.

Superlinear problems, analogous to \eqref{eq:toy-pb}, have already been studied. The case where  the gradient term exhibits a natural (i.e., quadratic) growth, namely
\begin{equation}\label{eq:quad}
\pat u-\D u=g(u)\av{\N u}^2+f,
\end{equation}
has been considered in \cite{DGSucraina, DGS} assuming different growths on $g(u)$. This equation is always superlinear if $g$ is positive and increasing. To handle this limit case we have available the Cole-Hopf change of unknown that allows to transform our equation into a semilinear one. In this case exponentially summable initial data have to be considered.

In the case where  $|u| \av{\N u}^q$ is replaced by $\av{\N u}^q$, we get the superlinear equation
\begin{equation}\label{eq:tesi}
\pat u-\D u=\av{\N u}^q+f, \qquad 1<q<2,
\end{equation}
for which we refer to \cite{M,BASW,BD}.
Another superlinear problem sharing some similarities with \eqref{eq:toy-pb} and \eqref{eq:tesi} is the  Cauchy-Dirichlet problem for the equation
\begin{equation}\label{eq:cesco}
\pat u-\D u=\frac{\av{\N u}^q}{\av{u}^\al}+f,\qq \al>0,\q 1+\al<q<2.
\end{equation}
In this case, the existence of solutions  is developed in \cite{MO}. The exponent $\al$ changes the superlinear range  to $1+\al<q<2$, and this means that the problem is superlinear for less values of $q$ that for $\al=0$. Both equations can be dealt with initial data belonging to certain Lebesgue spaces.

\vspace{-1mm}

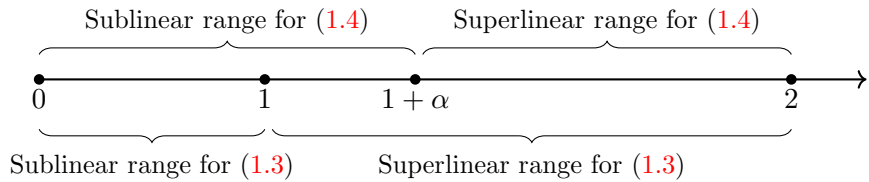
\begin{figure}[H]
\centering
\begin{tikzpicture}
\draw [->,thick] (0,0) -- (11,0);
\fill (0,0) circle (2pt) node[below]
{$0$};
\fill (3,0) circle (2pt) node[below] {$1$};

\fill (5,0) circle (2pt) node[below] {$1+\al$};

\draw[decorate,decoration={brace,amplitude=5pt,mirror,raise=4ex}]  (3.1,0) -- (10,0) node[midway,yshift=-3em]{\small Superlinear range for \eqref{eq:tesi}};
\draw[decorate,decoration={brace,amplitude=5pt,mirror,raise=4ex}]  (0,0) -- (3,0) node[midway,yshift=-3em]{\small Sublinear range for \eqref{eq:tesi}};

\draw[decorate,decoration={brace,amplitude=5pt,raise=2ex}]  (5.1,0) -- (10,0) node[midway,yshift=2em]{\small Superlinear range for \eqref{eq:cesco}};

\draw[decorate,decoration={brace,amplitude=5pt,raise=2ex}]  (0,0) -- (5,0) node[midway,yshift=2em]{\small Sublinear range for  \eqref{eq:cesco}};

\fill (10,0) circle (2pt) node[below] {$2$};

\end{tikzpicture}
\caption{\small Comparison between
the superlinear ranges  of $q$ in \eqref{eq:tesi}, \eqref{eq:cesco} (see \cite{M}, \cite{MO}, respectively).}
\end{figure}

Intuitively,  as the exponents of $\av{u}$ and $|\nabla u|$ increase in the right hand side of the equation,  one need to make stronger requirements on the data $u^0$ and $f$, that is, they have to be more integrable. Consequently, the assumptions on the data for problem \eqref{eq:toy-pb} should be stronger than those for problems \eqref{eq:tesi}, \eqref{eq:cesco}  (see \cite{M,MO}) and weaker than the assumptions for problem \eqref{eq:quad} (see  \cite{DGSucraina, DGS}). We will show that this idea is correct, in two different ways: first we prove an existence result with an exponentially integrable initial datum instead of a datum in a Lebesgue space. Then we prove a more unexpected and novel  existence result taking data in Lebesgue spaces, as studied in \cite{M}, but   assuming a smallness condition. Although both results are obtained by an approximation procedure, we get the respective estimates  with different methods. In fact, the case of exponentially integrable initial data is handled by using a logarithmic Sobolev inequality and a Gronwall-type lemma,  while the Gagliardo-Nirenberg inequality is essential to deal with less regular initial data.
 To study problem \eqref{eq:toy-pb} for very regular initial data, assumptions are similar to those needed to have existence of solutions in the case with natural growth $q=2$, although now the Cole-Hopf change of unknown cannot be applied. In this case, we may consider initial data belonging to a whole functional space (i.e., without assuming smallness). Nevertheless, it has the weakness of considering exponentionally summable initial data and not taking fully advantage of the fact that $q<2$ in \eqref{eq:toy-pb}.

To better understand our hypotheses on the data in the setting of Lebesgue spaces,
several remarks comparing with the results of \cite{M,MO} are in order. In \cite{M}, the regularity of the data which is required to find
solutions of problem \eqref{eq:tesi} is
\begin{equation}\label{ipotesiM}
u^0\in L^{\frac{N(q-1)}{2-q}}(\Omega), \qq f\in L^r(0,T;L^m(\Omega))\q\t{with}\q \frac{N}{m}+\frac{2}{r}\le \frac{q}{q-1},
\end{equation}
while, in \cite{MO},  the compatibility condition for problem \eqref{eq:cesco}  is given by
\begin{equation}\label{ipotesiMO}
u^0\in L^{\frac{N(q-1-\al)}{2-q}}(\Omega), \qq f\in L^r(0,T;L^m(\Omega))\q\t{with}\q \frac{N}{m}+\frac{2}{r}\le  \frac{q-2\al}{q-1-\al}.
\end{equation}
Since, for $\al>0$,
\[
\frac{N(q-1-\al)}{2-q}<\frac{N(q-1)}{2-q}
\qquad\mbox{and}\qquad
\frac{q-2\al}{q-1-\al}>
\frac{q}{q-1}\,,
\]
it follows that the assumptions \eqref{ipotesiMO} are weaker than \eqref{ipotesiM}, that is, one can prove the existence of the same kind of solution but with less regular data. In other words, %Therefore,
the term $|u|^{-\alpha}$ in \eqref{eq:cesco}  %exponent $\alpha$
has a regularizing effect.\\
Coming back to our main problem, % \eqref{eq:toy-pb},
the  compatibility condition for \eqref{eq:toy-pb} becomes (see \eqref{eq:f-0})
\[
u^0\in L^{\frac{Nq}{2-q}}(\Omega),\qq f\in L^r(0,T;L^m(\Omega))\q\t{with}\q\frac{N}{m}+\frac{2}{r}\le  \frac{q+2}{q}.
\]
Comparing with the assumptions  \eqref{ipotesiM}  and \eqref{ipotesiMO} on $u^0$, % of \cite{M, MO},
we clearly have the chain of inequalities
\[
\overbrace{\frac{Nq}{2-q}}^{\t{our case}}>\overbrace{\frac{N(q-1)}{2-q}}^{\t{case with $\av{\N u}^q$ \eqref{eq:tesi}}}>\overbrace{\frac{N(q-1-\al)}{2-q}}^{\t{case with $\av{\N u}^q{\av{u}^{-\al}}$ \eqref{eq:cesco}}}.
\]
As far as the source datum $f$ is concerned, the comparison is illustrated in Figure 2.
This means that, in \eqref{eq:toy-pb}, we can prove the existence of the same kind of solution of either \eqref{eq:tesi} or \eqref{eq:cesco} asking for more regular data than those in \eqref{eq:tesi} or in \eqref{eq:cesco}.

\begin{figure}[H]
 
\begin{tabular}{cc}
\begin{minipage}{.65\textwidth}
\begin{figure}[H]
\centering
\scalebox{.8}{

\begin{tikzpicture}[xscale=2,yscale=.68]

\coordinate (A) at (0,-1);
\coordinate (L) at (2,-1);
\coordinate (B) at (2,0);
\coordinate (C) at (1,5);
\coordinate (D) at (0,5);

\coordinate (F) at (2.8,1);
\coordinate (G) at (2,5);
\coordinate (M) at (2.8,-1);

\coordinate (H) at (3.6,2);
\coordinate (I) at (3,5);
\coordinate (N) at (3.6,-1);

\filldraw[fill=gray!30] (A) -- (N) -- (H) -- (I) -- (D) -- cycle;

\filldraw[ fill=gray!50] (A) -- (M) -- (F) -- (G) -- (D) -- cycle;
%\filldraw[ fill=gray!50, gray] (A) -- (M) -- (F) -- (G) -- (D) -- cycle;
\filldraw[fill=gray!70] (A) -- (L) -- (B) -- (C) -- (D)-- cycle;

\draw[ultra thick, color=gray!90]   (C) -- (B) -- (L);
\draw[ultra thick, color=gray!70] (G) --  (F) -- (M);
\draw[ultra thick, color=gray!50] (I) -- (H) -- (N);

\draw[->] (0,-1) -- (5,-1)  node[below right] {$\frac{1}{m}$};
\draw[->] (0,-1) -- (0,7) ;
\draw   (0,6) node[left] {$\frac{1}{r}$};

\draw (0,5) -- (5,5)   node[left] at (0,5) {$1$};
\fill (3,5) circle (0pt) node[above] {$\small \frac{2-q}{N(q-1-\al)}$};
\fill (0,2) circle (0pt) node[left] {$\small \frac{2-q}{N(q-1-\al)}$};
\fill (2,5) circle (0pt) node[above] {$\small \frac{2-q}{N(q-1)}$};
\fill (0,1) circle (0pt) node[left] {$\small \frac{2-q}{N(q-1)}$};
\fill (1,5) circle (0pt) node[above] {$\small \frac{2-q}{Nq}$};
\fill (0,0) circle (0pt) node[left] {$\small \frac{2-q}{Nq}$};

\fill (2,-1) circle (0pt) node[below ] {$\small \frac{N(2+q)-2(2-q)}{N^2q}$};

\fill (2.8,-2) circle (0pt) node[below ] {$\small \frac{Nq-2(2-q)}{N^2(q-1)}$};

\fill (3.6,-1) circle (0pt) node[below ] {$\small \frac{N(q-2\al)-2(2-q)}{N^2(q-1-\al)}$};

\draw [dashed] (0,0) -- (2,0);

\draw [dashed] (0,1) -- (2.8,1);

\draw [dashed] (0,2) -- (3.6,2);

\draw [dashed,->] (2.8,-2) -- (2.8,-1);
 
\end{tikzpicture}
}

\end{figure}
\end{minipage}
&
\begin{minipage}{.5\textwidth}
\vspace{5mm}
\small $f\in L^r(0,T;L^m(\Omega))$ with \\[10pt]
\begin{tikzpicture}
\draw[  fill=gray!70!,draw=gray!90,thick] (0,0) rectangle (.3,.3);
\end{tikzpicture}
\small   $\frac{N}{m}+\frac{2}{r}\le 1+\frac{2}{q}$ (for \eqref{eq:toy-pb})\\

\begin{tikzpicture}
\draw[  fill=gray!50,draw=gray!70,thick] (0,0) rectangle (.3,.3);
%\draw[  color=gray , pattern= north west lines, pattern color=gray] (0,0) rectangle (.3,.3);
\end{tikzpicture}
\small  $\frac{N}{m}+\frac{2}{r}\le \frac{q}{q-1}$ (for \eqref{eq:tesi})\\

\begin{tikzpicture}
\draw[  fill=gray!30, draw=gray!50,thick] (0,0) rectangle (.3,.3);
\end{tikzpicture}
\small  $\frac{N}{m}+\frac{2}{r}\le   \frac{q-2\al}{q-1-\al}$ (for \eqref{eq:cesco})\\

\end{minipage}

\end{tabular}

\caption{\small  Comparison between the  regularity of the source term $f$ in \eqref{eq:tesi}, \eqref{eq:cesco} (see \cite{M, MO}) and \eqref{eq:toy-pb} (see the precise assumptions in \eqref{eq:f-0b} below). Note that the three regions coincide for $q=2$.}
\end{figure}

Furthermore, the fact of having  $\av{u}$ multiplying the gradient power will weaken our existence results. Let us explain better this claim. When considering \eqref{eq:toy-pb},  we need to ask for small norm of the data, while in both \cite{M,MO},  no condition on the size of norms of $u^0$ or  $f$ was
necessary. Indeed, in those works,  the estimates are done on the level set function $G_k(u)=(\av{u}-k)_+\t{sign}(u)$ (see \eqref{tk1} and Fig.\ \ref{fig:GkTk}  below).
Roughly speaking, the techniques used in \cite{M,MO} can estimate
the integral in space and time containing the $q$-power of the gradient in terms of $\norm{G_k(u^0)}_{L^\si(\Omega)}$  times a quantity appearing in the left hand side so that, letting $k$ large enough, they can be absorbed.
 We will show that this trick is not possible in our case \eqref{eq:toy-pb} (see  Remark \ref{rmk:Gk} below).

Hence, in the study of the existence of finite energy and entropy solutions in the case with possibly non-regular Lebesgue data we assume they have small norms.
The core of our proof relies in the a priori estimates contained in Subsection \ref{sec:data-small}. To get these estimates, we take $u^0\in L^\sigma(\Omega)$ for certain $\sigma$ depending on $q$. Roughly speaking, $\displaystyle \sigma=\max\Big\{1,\, \frac{Nq}{2-q}\Big\}$. We must analyze several cases according to the exponent $q$. Indeed, observe that
\begin{align*}
&\si\ge2 &&\t{when}\qq\frac{4}{N+2}\le q<2,\\
&1<\si<2 &&\t{when}\qq\frac{2}{N+1}< q<\frac{4}{N+2},\\
&\si=1  &&\t{when}\qq0< q<\frac{2}{N+1}.
\end{align*}
We remark that the range $0< q<\frac{2}{N+1}$ leads to $\frac{Nq}{2-q}<1$, hence every $u^0\in L^1(\Omega)$ verifies the compatibility condition. On the other hand, if $\si\ge2$, we prove the existence of finite energy solutions, while we find entropy solutions as $\si<2$. \\
A particularly important case is the threshold $q=\frac{2}{N+1}$. This value of $q$ leads to $\frac{Nq}{2-q}=1$. However, as in the elliptic case (see \cite[Example 4.1]{GMP}), assuming $L^1$ data is not enough in order to obtain an existence result. A possibility could be taking $u^0\in L^{1+\eps}(\Omega)$, $\eps>0$, and then reasoning in a similar way as in the case of sharp data with $\si>1$ (see \cite[Theorem 5.6]{M}, \cite[Case with $(Q_2)$]{MO}).
Having in mind that $L^1$ data cannot be considered, in this work, we refine the data assumptions of the critical case  $q=\frac{2}{N+1}$ asking for more general data such that
\[
\int_\Omega |u^0|(\max\set{1,\log(\av{u^0})})^{\frac{N+1}2}dx \quad\mbox{and}\quad \iint_{Q_T} |f|\,(\max\set{1,\log(\av{f})})^{\frac{N+1}2}\,dx\,d\tau
\]
are (sufficiently) small. So, in this article, we improve  the results obtained in \cite{M, MO} for critical $q$.

As the exponent $q$ decreases from $2$ to $0$, the  gradient nonlinearity  gets weaker, which allows us to consider larger classes of data, i.e. data with less integrability. The price to pay to take such non-regular data is that solutions may have infinite energy.  However, even for small $q$, we can still consider more regular data such that we recover the finite energy formulation (see Proposition \ref{energy}).\\

This paper is organized as follows. Section 2 is devoted to preliminaries: notations, precise assumptions and statements of main results are included. Section \ref{sec:exp-bdd} contains the proof of existence of finite energy solutions when data are very regular. Section \ref{sec:small} is concerned with the proof of existence for data in Lebesgue spaces, satisfying a smallness condition. Section \ref{sec:energy} deals with the proof of existence of finite energy solutions for small values of $q$.

\section{Preliminaries}

\subsection{Notation}

We here introduce some functions which will be used in the sequel.\\
We define  the truncation function $T_k(s)$ and its primitive $\Theta_k(s)$, as well as the level set function $G_k(s)$  as
\begin{equation}\label{tk1}
 T_k(s) = \min\{|s|,k\}\,\sgn( s)\,, \quad  \Theta_k(s) = \int_0^s T_k(\eta)\,d\eta\,,\quad G_k(s)=s-T_k(s)\,,\quad
 s\in \R,\ k>0.
\end{equation}
\begin{figure}[H]
\begin{tabular}{cc}
\begin{minipage}{.45\textwidth}
\begin{figure}[H]
\centering
\begin{tikzpicture}[scale=.9]
\draw[->] (-3,0) -- (3,0) node[anchor=north west] {$s$};
\draw[->] (0,-2) -- (0,2) node[left] {{\color{gray}$T_k(s)$}};
% % % % % % % % % % % % % % % % % % % % % % % % %
\draw[very thick, gray] (-1,-1) -- (1,1);
\draw[very thick, gray] (1,1) -- (2,1);
\draw[very thick, gray] (-1,-1) -- (-2,-1);
\draw[very thick, dashed, gray] (2,1) -- (3,1);
\draw[very thick, dashed, gray] (-2,-1) -- (-3,-1);
% % % % % % % % % % % % % % % % % % % % % % % % %
\draw[dashed] (1,1) -- (1,0);
\draw[dashed] (-1,-1) -- (-1,0);
\draw[dashed] (1,1) -- (0,1);
\draw[dashed] (-1,-1) -- (0,-1);
% % % % % % % % % % % % % % % % % % % % % % % % %
\fill (1,0) circle (1.5pt) node[below] {$k$};

\fill (-1,0) circle (1.5pt) node[above] {$-k$};

\fill (0,1) circle (1.5pt) node[left] {$k$};
\fill (0,-1) circle (1.5pt) node[right] {$-k$};
\end{tikzpicture}
\end{figure}
\end{minipage}
&
\begin{minipage}{.45\textwidth}
\begin{figure}[H]
\centering
\begin{tikzpicture}[scale=.9]
\draw[->] (-3,0) -- (3,0) node[anchor=north west] {$s$};
\draw[->] (0,-2) -- (0,2) node[left] {{\color{gray!70!black}$G_k(s)$}};
% % % % % % % % % % % % % % % % % % % % % % % % %
\draw[very thick, gray!70!black] (-1,0) -- (1,0);
\draw[very thick, gray!70!black] (1,0) -- (2,1);
\draw[very thick, gray!70!black] (-1,0) -- (-2,-1);
\draw[very thick, dashed, gray!70!black] (2,1) -- (3,2);
\draw[very thick, dashed, gray!70!black] (-2,-1) -- (-3,-2);
% % % % % % % % % % % % % % % % % % % % % % % % %
% % % % % % % % % % % % % % % % % % % % % % % % %
\fill (1,0) circle (1.5pt) node[below] {$k$};

\fill (-1,0) circle (1.5pt) node[above] {$-k$};

\end{tikzpicture}
\end{figure}
\end{minipage}
\end{tabular}
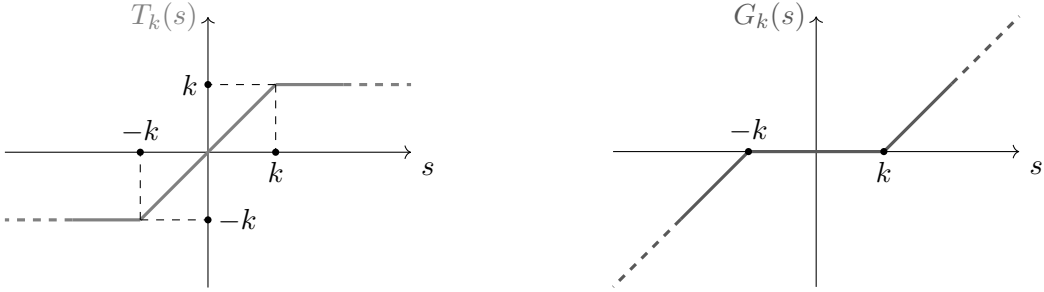
\caption{The functions    $T_k(s)$ and $G_k(s)$.}\label{fig:GkTk}
\end{figure}
\noindent
Furthermore, we set
\[
\log^*s= \max\{1, \log s\}\q\t{for}\q s>0,\qq \log^*0=1.
\]

We   denote by $\langle\cdot,\, \cdot \rangle$   the pairing between $H^{-1}(\Omega) + L^1(\Omega)$ and $H^1_0(\Omega)\cap L^\infty(\Omega)$.

We   name   either $c$ or $C$ constants (depending on the parameters of the problem but not on the approximation and regularization  parameters $n$ or $\eps$) that may vary from line to line.

We represent the $L^p(\Omega)$ norms of some function $f$ by either $\norm{f}_p$ or $\norm{f}_{L^p(\Omega)}$.

\subsection{General problem and assumptions}

The general problem we are going to consider is
\begin{equation}\label{eq:pb}
\begin{cases}
\begin{aligned}
\pat u-\dive{a(t,x,u,\N u)}&=H(t,x,u,\N u)&&\t{in } Q_T ,\\
u(0,x)&=u^0(x)&&\t{in }\Omega,\\
u(t,x)&=0&&\t{on }(0,T)\times\partial\Omega,
\end{aligned}
\end{cases}
\end{equation}
where
\[
a(t,x,s,\xi):Q_T\times \mathbb{R}\times \mathbb{R}^N\to  \mathbb{R}^N
,\qq
H(t,x,s,\xi):Q_T\times \mathbb{R}\times \mathbb{R}^N\to \mathbb{R}
\]
 are Caratheodory functions (i.e., measurable with respect to $(t,x)$ for every $(s,\xi)\in \R\times \R^N$, and continuous in $(s,\xi)$ for a.e.\ $(t,x)\in Q_T$).\\
Furthermore, we assume that     $a(t,x,s,\xi)$ and $H(t,x,s,\xi)$ are such that
\begin{itemize}
\item  the classical Leray-Lions structure conditions hold:
\begin{subequations}\label{eq:A}
\begin{equation}\label{A1}
\exists \,\alpha>0:\quad
\alpha|\xi|^2\le a(t,x,s,\xi)\cdot\xi,
\end{equation}
\begin{equation}\label{A2}
\exists \,\beta>0:\quad|a(t,x,s,\xi)|\le \beta[|s| +|\xi| +h(t,x)]\quad\text{where }\,\, h\in L^{2}(Q_T),
\end{equation}
\begin{equation}\label{A3}
 (a(t,x,s,\xi)-a(t,x,s,\eta))\cdot(\xi-\eta)> 0,
\end{equation}
\end{subequations}
for almost every $(t,x)\in Q_T$, for every $s\in \mathbb{R}$ and for every $\xi$, $\eta$ in $\mathbb{R}^N$, $\xi\ne\eta$;
\item the right hand side satisfies the growth condition:
\begin{equation}\label{H}
\ds
\begin{array}{c}
\exists \,\,\gamma>0:\,\,
|H(t,x,s,\xi)|\le \gamma \av{s}|\xi|^q+f(t,x)\qq 0<q<2,
\end{array}
\end{equation}
for almost every $(t,x)\in Q_T$, for every $(s,\xi)\in\R\times\mathbb{R}^N$ and for some measurable function $f(t,x)$ on $Q_T$, for which some integrability assumptions will be given hereafter. %forcing term $f$
\end{itemize}
For the sake of simplicity we will assume that
\begin{equation}\label{alfagamma1}
\al=\ga=1.
\end{equation}
The general case can be adapted with few modifications, although the smallness condition may change accordingly.\\
We delay data assumptions to the statements of the main theorems.

\subsection{Definitions of solution}

 We will now give two definitions of solutions  to problem \eqref{eq:pb}:

\begin{definition}[Finite energy solutions]\label{defsol}
We say that a real valued  function $u$ is a finite energy solution to problem \eqref{eq:pb} if
\[
u\in L^2(0,T;H_0^1(\Omega))\cap C([0,T]; L^2(\Omega))\,,\quad a(t,x,u,\N u)\in L^2(Q_T)\,,\quad H(t,x,u,\N u) \in L^1(Q_T),
\]
 and it verifies the weak formulation
\begin{align}
&\intO u(t,x)v(t,x)\,dx-\intO u^0(x)v(0,x)\,dx- \int_0^t \langle\pa_\tau v,u \rangle\,d\tau+\iint_{Q_t} a(\tau,x,u,\N u)\cdot\N v\,dx\,d\tau \nonumber \\
&=\iint_{Q_t}H(\tau,x,u,\N u) v\,dx\,d\tau \label{formdeb}
\end{align}
for every $t\in (0,T]$ and for every test function $v(\tau,x)$ such that
\[
v\in L^2(0,t;H_0^1(\Omega))\cap L^\infty(Q_t)\cap C([0,t];L^2(\Omega)),\q \pa_\tau v\in L^2(0,t;H^{-1}(\Omega)).
\]
\end{definition}

\smallskip

\begin{definition}[Infinite energy solutions - Entropy solutions] \label{entrsol}
We say that a real valued function $u$ is an entropy solution to problem \eqref{eq:pb} if
\[
u\in C([0,T];L^1(\Omega)), \qquad
T_k(u)\in L^2(0,T;H_0^1(\Omega))\quad\forall k>0,\qquad H(t,x,u,\nabla u)\in L^1(Q_T),
\]
and
\begin{align*}
&\intO \Theta_k(u(t,x)-v(t,x))\,dx-\intO \Theta_k(u^0(x)-v(0,x))\,dx+\int_{0}^t\langle\pa_\tau v, T_k(u-v)\rangle\,d\tau\\
&\q+\iint_{Q_t}a(\tau,x,u,\N u)\cdot\N T_k(u-v)\,dx\,d\tau=\iint_{Q_t}H(\tau,x,u,\N u) T_k(u-v)\,dx\,d\tau
\end{align*}
for every $t\in(0,T]$, for every $k>0$ and for every test function $v(\tau,x)$ such that
\[
v\in L^2(0,t;H_0^1(\Omega))\cap L^\infty(Q_t)\cap C([0,t];L^1(\Omega)),\q \pa_\tau v\in L^2(0,t;H^{-1}(\Omega))+L^1(Q_t).
\]
\end{definition}

Note that, under the assumptions made, all the integral  terms in the previous definitions make  
sense.

The concept of entropy solutions was introduced by Benilan et al. in \cite{BBV}  to deal with elliptic equations with $L^1$ data. The adaptation for parabolic problems was done by Prignet and Andreu et al.  in \cite{Pr, AMST} respectively. An equivalent formulation is the one of the renormalized solutions, whose definition can be found in  \cite{BM97} by Blanchard and Murat.

\subsection{Statement of the results}
We are going to consider two main frameworks. \\
First, we prove the following existence result which is global in time and with exponentially integrable initial datum.
\begin{theorem}[Global existence for exponentially integrable
solutions]\label{teo:global-bdd}
Assume that \eqref{eq:A} and \eqref{H} are in force. We also assume that the data $u^0(x)$ and $f(t,x)$ satisfy
\begin{gather}
\mbox{there exists $\delta>0$ such that}\quad \int_\Omega e^{\delta (u^0)^2}\,dx < \infty\,;\label{ipot-exp-u0}%\\
\end{gather}
\begin{equation}\label{ipot-exp-f}
\begin{array}{c}
f\in L^r (\log L)^{r/2}\big(0,T;L^m(\Omega)\big),\\[.3em]
\mbox{for some $r,m$ such that\,\,
$\ds \frac{N}{m}+\frac{2}{r}=2$,\,\, $1<r\le 2,\,\, N\le m<\infty$.}
\end{array}
\end{equation}

Then there exists a finite energy solution $u$ of problem \eqref{eq:pb} such that
\begin{equation}\label{regolarexp}
e^{\frac\delta2 u^2}-1 \in L^\infty(0,T;L^2(\Omega)) \cap L^2(0,T;H_0^1(\Omega))\,.
\end{equation}
Moreover, $H(t,x,u,\nabla u)u\in L^1(Q_T)$, and the following energy identity holds for every $t\in (0,T]$:
\begin{equation}\label{energy}
\frac12\int_\Omega u(t,x)^2 dx-\frac12\int_\Omega u^0(x)^2 dx+\iint_{Q_t}a(\tau,x,u,\nabla u)\cdot \nabla u\, dx\, d\tau=\iint_{Q_t}H(\tau,x,u,\nabla u) u\, dx\, d\tau\,.
\end{equation}

\end{theorem}

We observe that \eqref{ipot-exp-u0} is the same as requiring that $u^0(x)$  belongs  to the Orlicz space associated to the $N$-function $\eta(s)=e^{s^2}-1$ (see \cite{RR,K}).

We note that assumptions on $f$  do not depend on $q$. Actually, this requirement corresponds to \eqref{eq:f-0b} below when $q=2$.

In the next result, we remove the strong integrability assumptions on the data, looking for solutions which belong to Lebesgue spaces, but in this case we impose  a smallness condition on the norm of the data, obtaining the following  existence result.

\begin{theorem}[Global existence with smallness condition on the norm of unbounded data]\label{teo:small}
Assume that \eqref{eq:A} and \eqref{H} are in force, and that    the data $u^0(x)$ and $f(t,x)$ satisfy the following hypotheses:
\begin{itemize}
\item   when $ \frac{2}{N+1}<q<2$,
\begin{subequations}\label{eq:f-0}
\begin{align}
u^0\in L^\si(\Omega)&\,\,\t{with}\,\,\si=\frac{Nq}{2-q},\label{eq:f-0a}\\
 f\in L^r(0,T;L^m(\Omega))&\,\,\t{with}\,\,\frac{N}{m}+\frac{2}{r}= 1+\frac{2}{q}\,,\q 1\le r\le \sigma'\,,\q \frac{N^2q}{N(q+2)-2(2-q)}\le m\le \sigma\,;\label{eq:f-0b}
\end{align}
\end{subequations}

\item  when $q=\frac{2}{N+1}$,
\begin{equation*}
\int_\Omega |u^0|(\log^*|u^0|)^{\frac{N+1}2}\,dx<\infty,\qq \iint_{Q_T} |f|\,(\log^*|f|)^{\frac{N+1}2}\,dx\,d\tau<\infty;
\end{equation*}
\item   when $0< q<\frac{2}{N+1}$,
\begin{equation*}
u^0\in L^1(\Omega),\qq f\in L^1(Q_T).
\end{equation*}
\end{itemize}
In each case, we assume that their norms are sufficiently small (see \eqref{smallness}, \eqref{smallness0}, \eqref{smallness-critico}, \eqref{eq:smallness-1} below).
Then there exists a  solution $u$ of problem \eqref{eq:pb} such that
\begin{itemize}
\item $u$ is a finite energy solution  when $ \frac{4}{N+2}\le q<2$, and \eqref{energy} holds;
\item   $u$ is an entropy solution when $0< q <\frac{4}{N+2}$.
\end{itemize}
Moreover, in the case $ \frac{2}{N+1}<q<2$, the solution becomes more regular satisfying $u\in C([0,T];L^\sigma(\Omega))$.
\end{theorem}

The connection between the power of the gradient term and the space of initial data is illustrated in the following figure.
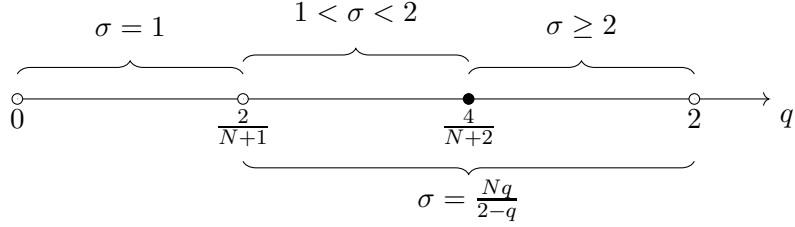
\begin{figure}[H]
\centering
\begin{tikzpicture}
\draw [->,thin] (0,0) -- (10,0);
\node[circle,draw=black, fill=white, inner sep=0pt,minimum size=4pt] at (0,0) {};
\fill (0,0) circle (0pt) node[below] {$0$};
\node[circle,draw=black, fill=white, inner sep=0pt,minimum size=4pt] at (3,0) {};
\fill (3,0) circle (0pt) node[below] {$\frac{2}{N+1}$};
\node[circle,draw=black, fill=black, inner sep=0pt,minimum size=4pt] at (6,0) {};
\fill (6,0) circle (2pt) node[below] {$\frac{4}{N+2}$};

  \node[circle,draw=black, fill=white, inner sep=0pt,minimum size=4pt] at (9,0) {};
\fill (9,0) circle (0pt) node[below] {$2$};
  \node[circle,draw=black, fill=white, inner sep=0pt,minimum size=0pt] at (9,0) {};
\fill (10,0) circle (0pt) node[below right] {$q$};
\draw [decorate,decoration={brace,amplitude=5pt,mirror,raise=5ex}]
  (3,0) -- (9,0) node[midway,yshift=-3.5em]{$\si=\frac{Nq}{2-q}$};
\draw [decorate,decoration={brace,amplitude=5pt,raise=2ex}]
  (6,0) -- (9,0) node[midway,yshift=2.5em]{$\si\ge2$};
\draw [decorate,decoration={brace,amplitude=5pt,raise=3ex}]
  (3,0) -- (6,0) node[midway,yshift=3em]{$1<\si<2$};
 \draw [decorate,decoration={brace,amplitude=5pt,raise=2ex}]
  (0,0) -- (3,0) node[midway,yshift=2.5em]{$\si=1$};
\end{tikzpicture}
\end{figure}

We briefly explain why we have chosen exactly the value $\si=Nq/(2-q)$ when $q>2/(N+1)$, i.e. $\si>1$. Suppose to have a problem of the type
\[
\pat u-\D u=\av{u}^\lm\av{\N u}^q\qq\t{with}\q 0< \lm<1.
\]
Then, reasoning as in \cite[Appendix C]{M} (see also \cite[Section 6]{MO}), we can prove that the lowest Lebesgue regularity on the initial data that allows to obtain an a priori estimate is
\[
\si_\lm=\frac{N(q-1+\lm)}{2-q}\qq\t{with}\q q>1-\lm.
\]
The argument used in the aforementioned works fails when $\lm$ is exactly $1$, due to the criticality of this case. However, we can justify the value $\si=Nq/(2-q)$  as the limit of $\si_\lm$ as $\lm\to 1^+$.

We finally introduce our result on finite energy solutions for small values of $q$ when the data $u^0$ and $f$ have greater integrability than in Theorem \ref{teo:small}.

\begin{proposition}[Global existence of energy solutions with small $q$]\label{prop:energy}
Assume that \eqref{eq:A} and \eqref{H} are in force, and that  the growth parameter $q$  verifies
\[
0<q<\frac{4}{N+2}.
\]
Assume also that  the data $u^0(x)$ and $f(t,x)$ satisfy
\begin{align*}
u^0\in L^2(\Omega),\qq f\in L^r(0,T;L^m(\Omega))\q\t{with}\q\frac{N}{m}+\frac{2}{r}= \frac{N+4}{2}\,,\q 1\le r\le 2\,,\q \frac{2N}{N+2}\le m\le 2\,.
\end{align*}
We assume that their norms are sufficiently small.
Then there exists a   finite energy solution  $u$ of problem \eqref{eq:pb}.

\end{proposition}

\subsection{The approximating problem}
We consider the following approximating problem:
\begin{equation}\label{eq:pbn}
\begin{cases}
\begin{aligned}
\pat u_n-\dive{a(t,x,u_n,\N u_n)}&=T_n\pare{H(t,x,u_n,\N u_n)}&&\t{in }Q_T,\\
u_n(0,x)&=T_n(u^0(x))&&\t{in }\Omega,\\
u_n(t,x)&=0&&\t{on }(0,T)\times\partial\Omega.
\end{aligned}
\end{cases}
\end{equation}
The approximate solution $u_n$ exists (see \cite[Theorem 1.1 and Comment 1.2]{BMP} and \cite[Theorem 1.1]{P99}) and verifies, for every fixed $n\in\mathbb{N}$,
\[
u_n\in L^2(0,T;H_0^1(\Omega))\cap L^\infty(Q_T)\cap C([0,T];L^q(\Omega)) \q \forall\,q<\infty,\q \pat u_n\in L^2(0,T;H^{-1}(\Omega)),
\]
and also the weak formulation
\begin{align}\label{eq:sol-pbn}
\intO u_n(t,x)v(t,x)\,dx&-\intO T_n(u^0(x))v(0,x)\,dx-\int_{0}^t\langle\pa_\tau v, u_n\rangle\,d\tau
+\iint_{Q_t}a(\tau,x,u_n,\N u_n)\cdot\N v\,dx\,d\tau\nonumber\\
&=\iint_{Q_t}T_n\pare{H(\tau,x,u_n,\N u_n)}v\,dx\,d\tau
\end{align}
for every $t\in (0,T]$ and for every test function $v(\tau,x)$ such that
\[
v\in L^2(0,t;H_0^1(\Omega))\cap L^\infty(Q_t),\q \pa_\tau v\in L^2(0,t;H^{-1}(\Omega))+L^1(Q_t).
\]
Moreover, for every continuous and piecewise $C^1$ function $\varphi$  such that $\varphi(0)=0$, one has
\begin{align*}
\intO \Phi(u_n(t,x))\,dx &- \intO \Phi(T_n(u^0(x)))\,dx + \iint_{Q_t}a(\tau,x,u_n,\N u_n)\cdot\N u_n\,\varphi'(u_n)\,dx\,d\tau \nonumber\\
&=\iint_{Q_t}T_n\pare{H(\tau,x,u_n,\N u_n)}\varphi(u_n)\,dx\,d\tau,
\end{align*}
where $\Phi(s)$ is a primitive function of $\varphi(s)$ on $\R$; this implies that the function $t\mapsto \intO \Phi(u_n(t,x))\,dx$ is absolutely continuous in $[0,T]$, and that
\begin{align}
\frac{d}{dt}&\intO \Phi(u_n(t,x))\,dx + \int_\Omega a(t,x,u_n(t,x),\N u_n(t,x))\cdot\N u_n(t,x)\,\varphi'(u_n(t,x))\,dx \nonumber\\
&=\int_{\Omega}T_n\pare{H(t,x,u_n(t,x),\N u_n(t,x))}\varphi(u_n(t,x))\,dx \qquad\mbox{for a.e.\ $t\in[0,T]$}\label{ftestu'}
\end{align}
(see, for instance, \cite[Lemma 2.5]{Ppower}).

\section{Proof of Theorem \ref{teo:global-bdd} - Global existence with exponentially integrable initial datum} 
\label{sec:exp-bdd}

\subsection{A priori estimates}

In the proof of Theorem \ref{teo:global-bdd}, we will use the following logarithmic Sobolev inequality (see \cite{DGS}), which is of the same kind introduced by Gross in \cite{Gr}.

\begin{proposition}\label{prop:adams}
{\bf (Logarithmic Sobolev inequality)}
For every $\kappa>0$,
there exists a
positive constant $\overline C$  (depending on $N$, $|\Omega|$, $\kappa$)
 such that, for every $\eps>0$ and for every $v\in H^1_0(\Omega)$,
$$
\int_\Omega |v(x)|^2 \pare{\log^* |v(x)|}^\kappa \,dx \leq
\overline C\,\bigg[\eps\int_\Omega |\nabla
v(x)|^2 \,dx + \normaL{v}{2}^2\bigg(\log^{*}\bigg(\frac{1}{\eps}\bigg)\bigg)^\kappa+
\normaL{v}{2}^2 \big(\log^*\normaL{v}{2}\big)^\kappa\bigg] \,.
$$

\end{proposition}

We define the following real functions:
\begin{equation}\label{PhiPsi}
\Phi(s)=e^{\delta s^2}\,,\qquad \varphi(s)= \Phi'(s)=2\delta s e^{\delta s^2}\,,\qquad \Psi(s)=\int_0^s\sqrt{\Phi''(z)}\, dz\,,
\end{equation}
where $\delta$ is the same as in \eqref{ipot-exp-u0}. Note that, using de l'H\^opital's rule,
\begin{equation}\label{ec:3.00}
\Psi(s)^2\sim 4\Phi(s)= 4 e^{\delta s^2}\qquad \hbox{as }s\to\infty\,, \qquad \Psi(s)^2\sim 2\delta s^2\qquad \hbox{as }s\to0\,.
\end{equation}
It follows that
\begin{equation}\label{pp}
\Psi(s)^2 \le c(\delta) \Phi(s)
\end{equation}
and that $\Psi(u^0)\in L^2(\Omega)$.

\begin{proposition}\label{prop:stimaprioriexp}
Under the assumptions of Theorem \ref{teo:global-bdd}, the solutions $u_n$ of the approximate problems \eqref{eq:pbn} satisfy
\begin{equation}\label{stimaprioriexp}
\normaL{\Psi(u_n)}{L^\infty(0,T;L^2(\Omega))}^2 + \normaL{\nabla\Psi(u_n)}{L^2(Q_T)}^2 \le C\,,
\end{equation}
where $C>0$ is a constant which depends on $N$, $|\Omega|$ and on the data through the values $\int_\Omega e^{\delta (u^0)^2}\,dx$ and $\int_0^T \normaL{f(t)}{m}^r\left(\log^*{\normaL{f(t)}{m}}\right)^{r/2}dt$.
Moreover, the following estimate holds:
\begin{equation}\label{altrastimapriori0}
\iint_{Q_T}\big |T_n(H(t,x,u_n,\nabla u_n))\,\varphi(u_n)\big| \,dx\,dt\le C\,.
\end{equation}
\end{proposition}

\begin{proof}
We take $\varphi(u_n)$ (as defined in \eqref{PhiPsi})
 in \eqref{ftestu'}, obtaining, for a.e. $t\in(0,T)$,
\begin{align}
\frac d{dt}\int_\Omega \Phi(u_n)\,dx+\int_\Omega \varphi'(u_n)|\nabla u_n|^2\,dx & \le \int_\Omega\big |T_n(H(t,x,u_n,\nabla u_n))\,\varphi(u_n)\big| \,dx \nonumber\\
&\le \int_\Omega 2\delta u_n^2 e^{\delta u_n^2}|\nabla u_n|^q\,dx+\int_\Omega 2\delta |u_n|
e^{\delta u_n^2}|f|\,dx\,,\label{ec:3.01}
\end{align}
where we have used the assumptions \eqref{A1}, \eqref{H} and \eqref{alfagamma1}.\\
We now examine the terms in \eqref{ec:3.01}. First notice that, by definition of $\Psi$,
\begin{equation}\label{ec:3.02}
\int_\Omega \varphi'(u_n)|\nabla u_n|^2\,dx = \int_\Omega |\nabla \Psi(u_n)|^2\,dx\,.
\end{equation}
As far as the right hand side is concerned, we observe that $2\delta u_n^2 e^{\delta u_n^2}\le  (2\delta)^{-1} \varphi'(u_n)$.
Thus, Young's inequality and \eqref{ec:3.00} lead to
\begin{align*}
&\int_\Omega 2\delta u_n^2 e^{\delta u_n^2}|\nabla u_n|^q\,dx\le \eps \int_\Omega \varphi'(u_n)|\nabla u_n|^2\,dx+c(\eps) \int_\Omega \delta u_n^2 e^{\delta u_n^2}\,dx\\
&\le\eps  \int_\Omega |\nabla \Psi(u_n)|^2\,dx+c(\eps) \bigg[\int_\Omega \Psi(u_n)^2 \log^*{\Psi(u_n)}\,dx\  +\, 1\bigg]\,.
\end{align*}
On account of \eqref{ec:3.02}, we apply this inequality to write \eqref{ec:3.01} as
\begin{equation}\label{ec:3.03}
\frac d{dt}\int_\Omega \Phi(u_n)\,dx+\int_\Omega |\nabla \Psi(u_n)|^2\,dx\le c\int_\Omega \Psi(u_n)^2 \log^*{\Psi(u_n)}\,dx+c\int_\Omega \delta |u_n| e^{\delta u_n^2}|f|\,dx\, + c.
\end{equation}
Applying Proposition \ref{prop:adams} with $\kappa=1$, we obtain
\begin{equation*}
\int_\Omega \Psi(u_n)^2 \log^*{\Psi(u_n)}\,dx\le
 \eps\int_\Omega |\nabla
\Psi(u_n)|^2 \,dx + c(\eps)  \normaL{\Psi(u_n)}2^2 + c
\normaL{\Psi(u_n)}2^2 \log^*\normaL{\Psi(u_n)}2\,.
\end{equation*}
Absorbing again the first term, \eqref{ec:3.03} becomes
\begin{align}\label{ec:3.04}
\frac d{dt}\int_\Omega \Phi(u_n)\,dx+\int_\Omega |\nabla \Psi(u_n)|^2\,dx&\le c\bigg(\int_\Omega \Psi(u_n)^2\,dx\bigg) \log^*\bigg(\int_\Omega \Psi(u_n)^2\,dx\bigg)\nonumber\\
&\q+c\pare{\int_\Omega   |u_n| e^{\delta u_n^2}|f|\,dx+1}.
\end{align}
Finally, we have to study the last integral  of the right hand side. We begin by noting that
\[\int_\Omega   |u_n| e^{\delta u_n^2}|f|\,dx\le c \int_\Omega |f| \big(1+\Psi(u_n)^2 \left(\log^*{\Psi(u_n)}\right)^{1/2}\big)\,dx.\]
Since
\[\frac1m+\frac{2}{2^*r'}+\frac1r=
1\,,\]
we may apply H\"older's inequality obtaining
\begin{align*}
\int_\Omega |f| \Psi(u_n)^2 \left(\log^*{\Psi(u_n)}\right)^{1/2}\,dx \le  \normaL{f(t)}{m}\normaL{\Psi(u_n)}{2^*}^{2/r'}\left(\int_\Omega \Psi(u_n)^2\left(\log^*{\Psi(u_n)}\right)^{r/2}\,dx\right)^{1/r}.
\end{align*}
We next apply Young's  and then Sobolev's inequalities to get
\begin{align*}
\int_\Omega |f| \Psi(u_n)^2 \left(\log^*{\Psi(u_n)}\right)^{1/2} \,dx&\le \eps \normaL{\Psi(u_n)}{2^*}^{2}\,dx+C(\eps)
\normaL{f(t)}{m}^r\int_\Omega \Psi(u_n)^2\left(\log^*{\Psi(u_n)}\right)^{r/2}\,dx\\
&\le \eps c \normaL{\nabla \Psi(u_n)}2^{2}\,dx+C(\eps)
\normaL{f(t)}m^r\int_\Omega \Psi(u_n)^2\left(\log^*{\Psi(u_n)}\right)^{r/2}\,dx.
\end{align*}
Going back to \eqref{ec:3.04}, if $\eps$ is small enough the first term on the right hand side can be absorbed, therefore
\begin{align}\label{ec:3.05}
\frac d{dt}\int_\Omega \Phi(u_n)\,dx+\int_\Omega |\nabla \Psi(u_n)|^2\,dx
&\le c\bigg(\int_\Omega \Psi(u_n)^2\,dx\bigg) \log^*\bigg(\int_\Omega \Psi(u_n)^2\,dx\bigg)+ c\pare{\normaL{f(t)}1+1} \nonumber\\
&\q
+c\normaL{f(t)}m^r\int_\Omega \Psi(u_n)^2\left(\log^*{\Psi(u_n)}\right)^{r/2}\,dx.
\end{align}
Applying Proposition \ref{prop:adams} with $\kappa=r/2$ to the last term, we obtain
\begin{align*}
&c\normaL{f(t)}m^r\int_\Omega \Psi(u_n)^2\left(\log^*{\Psi(u_n)}\right)^{r/2}\,dx\\
&\le
\eps c \normaL{f(t)}m^r \int_\Omega |\nabla \Psi(u_n)|^2+c\normaL{f(t)}m^r\bigg[\big(\log^*(1/\eps)\big)^{r/2}+\bigg(\log^*\pare{\int_\Omega\Psi(u_n)^2}\bigg)^{r/2}\bigg]\int_\Omega\Psi(u_n)^2\,dx.
\end{align*}

In order to absorb the first term on the right hand side, we  choose $2\eps=\frac1{c \|f(t)\|_m^r}$, so that \eqref{ec:3.05} becomes
\begin{align*}%
&\frac d{dt}\int_\Omega \Phi(u_n)\,dx+\int_\Omega |\nabla \Psi(u_n)|^2\,dx\nonumber\\
&\le c  \bigg[ \bigg(\int_\Omega \Psi(u_n)^2\,dx\bigg) \log^*\left(\int_\Omega \Psi(u_n)^2\,dx\right) +  \normaL{f(t)}1 + \normaL{f(t)}m^r\pare{1+\log^*\normaL{f(t)}m}^{r/2}\int_\Omega \Psi(u_n)^2\,dx\nonumber\\
&\qq+\normaL{f(t)}m^r\int_\Omega \Psi(u_n)^2\left(\log^*\pare{\int_\Omega \Psi(u_n)^2\,dx}\right)^{r/2}\,dx+ 1
\bigg]   \nonumber\\
&\le c\left[1+\normaL{f(t)}m^r\pare{1+\log^*\normaL{f(t)}m}^{r/2}\right]\bigg(\int_\Omega \Psi(u_n)^2\,dx\bigg) \log^*\left(\int_\Omega \Psi(u_n)^2\,dx\right) + c\pare{\normaL{f(t)}1 + 1},
\end{align*}
where we have used the inequality $r\le 2$, which follows from the assumptions. Denoting
\[
\xi_n(t)=\int_\Omega \Phi(u_n)\,dx,
\quad \Upsilon(t)=c\left[1+\normaL{f(t)}m^r\pare{1+\log^*\normaL{f(t)}m}^{r/2}\right], \quad \Theta(t) = c\big(\!\normaL{f(t)}1+1\big),
\]
and having in mind \eqref{pp} and the assumption \eqref{ipot-exp-f} on $f$, we deduce that
\[\xi_n'(t)\le \Upsilon(t)\xi_n(t)\log^*\xi_n(t) + \Theta(t)\le \max\{\Upsilon(t),\Theta(t)\}\big(\xi_n(t)\log^*\xi_n(t)+1\big),\]
with $\Upsilon, \Theta\in L^1(0,T)$. It yields that $\xi_n(t)$ is bounded in $[0,T]$. Integrating we get, for every $t\in[0,T]$,
\begin{equation*}
\int_\Omega \Phi(u_n(t))
\,dx+\iint_{Q_t} |\nabla \Psi(u_n)|^2\,dx\,d\tau
\le c(T)+\int_\Omega \Phi(u_{0})
\,dx.
\end{equation*}
By \eqref{pp}, estimates of $\Psi(u_n)$ both in $L^\infty(0,T; L^2(\Omega))$ and in $L^2(0,T;H_0^1(\Omega))$ follow, as well as \eqref{altrastimapriori0}.

\end{proof}

Recalling  that the function $\Psi$ has an exponential growth (see \eqref{ec:3.00}),  we deduce from Proposition \ref{prop:stimaprioriexp} the following result.
\begin{corollary}\label{Cor:1}
The sequence of approximate solutions $\set{u_n}$ is bounded in
\[
L^2(0,T;H_0^1(\Omega))\cap L^\infty(0,T; L^p
(\Omega)) \quad\mbox{for all $p<\infty$.}
\]
\end{corollary}

We point out that $\set{\Psi(u_n)}$ is bounded in $C([0,T]; L^2(\Omega))$,
therefore for every $p$ there exists a constant $C(p)>0$ such that for {\it every} (not only {\it almost every}) $t\in [0, T]$, then
\begin{equation*}%
\int_\Omega |u_n(t,x)|^p dx\le C(p) \qquad\hbox{ for all } n\in\bN.
\end{equation*}

\subsection{Compactness results}

\begin{proposition}\label{convS}
There exists a subsequence (still denoted by $\{u_n\}$) and a function $u\in L^2(0,T; H_0^1(\Omega))\cap L^\infty(0,T; L^p(\Omega))$ (for all $p<\infty$) such that:
\begin{align}
u_n&\to u && \t{a.e. in $Q_T$ and strongly in $L^p(Q_T)$, for all $p<\infty$} ,\label{Conv:1-exp}\\
\N u_n&\rightharpoonup \N u && \t{weakly in } (L^2(Q_T))^N,\label{Conv:3-exp} \\
\nabla T_k(u_n) &\to \nabla T_k(u) && \hbox{strongly in }(L^2(Q_T))^N\quad\forall k>0,\label{Conv:1.25-exp}\\
\N u_n&\to \N u && \t{a.e. in } Q_T,\label{Conv:1.1-exp}\\
T_n\pare{H(t,x,u_n,\N u_n)}&\to H(t,x,u,\N u)   && \t{strongly in } L^1(Q_T),\label{Conv:1.2-exp}\\
a(t,x,u_n,\nabla u_n)&\to a(t,x,u,\nabla u)&&\t{a.e. in }Q_T\,, \label{Conv:1.3-exp}\\
a(t,x,u_n,\nabla u_n)&\rightharpoonup a(t,x,u,\nabla u)&&\t{weakly in }(L^2(Q_T))^N,\label{Conv:1.45-exp}\\
u_n &\to u && \hbox{strongly in }C([0,T]; L^1(\Omega)).\label{Conv:1.3bis-exp}
\end{align}
\end{proposition}

\begin{proof}
Corollary \ref{Cor:1} implies that one can extract a subsequence   such that
\begin{equation*}%
  \N u_n \rightharpoonup \N u\quad \mbox{weakly in $L^2(Q_T)$.}
\end{equation*}
We want to show that
\begin{equation} \label{equi-int}
\mbox{the sequence $\big\{T_n(H(t,x,u_n,\N u_n))\big\}$ is equi-integrable in $Q_T$. }
\end{equation}
Indeed, for every measurable $E\subseteq Q_T$, we have that
\begin{align*}
\iint_{E} &\big|T_n\pare{H(t,x,u_n,\N u_n)}\big| \,dx\,dt \\
 &\le
\frac1{\varphi(k)} \iint_{\{|u_n|>k\}} \big|T_n(H(t,x,u_n,\N u_n))\varphi(u_n)\big| \,dx\,dt
+ \iint_{E\cap \{|u_n|\leq k\}} \big(|u_n|\,|\nabla u_n|^q +f(t,x)\big) \,dx\,dt \\
 &\le
\frac1{\varphi(k)} \underbracket[0.14ex]{\iint_{Q_T} \big|T_n(H(t,x,u_n,\N u_n))\varphi(u_n)\big| \,dx\,dt}_{\le C \mbox{ by \eqref{altrastimapriori0}}}
+ k |E|^\frac{2-q}{2}\underbracket[0.14ex]{\bigg[\iint_{Q_T} |\nabla u_n|^2 \,dx\,dt\bigg]^{q/2}}_ {\le C \mbox{ by \eqref{stimaprioriexp}}}\!\!\!+ \iint_E |f(t,x)|\,dx\,dt,
\end{align*}
where $\vp$ has been defined in \eqref{PhiPsi}.
Therefore, for any fixed $\eps>0$, we can fix $k$ large enough such that the first addend is smaller than $\eps/3$, and then fix $|E|$ small enough such that the remaining two terms are each smaller than $\eps/3$. Thus \eqref{equi-int} holds.
In particular, $T_n(H(t,x,u_n,\N u_n))$ is bounded in $L^1(Q_T)$.

Then the sequence $\set{\partial_tu_n}$ is bounded in $L^2(0,T;H^{-1}(\Omega))+L^1(Q_T)$, and we can apply  \cite[Corollary 4]{Simon} to conclude that $u_n\to u$ strongly in $L^2(Q_T)$. By Corollary \ref{Cor:1}, the convergence is also strong in $L^p(Q_T)$, for every $p<\infty$.
Using \eqref{equi-int}, it is possible to find a subsequence and a function $g\in L^1(Q_T)$ such that
\begin{equation}\label{deboleL1}
   T_n(H(t,x,u_n,\N u_n)) \rightharpoonup g \qquad\mbox{weakly in $L^1(Q_T)$.}
\end{equation}
Therefore one can apply \cite[Theorem 4.1]{BM} to conclude that \eqref{Conv:1.25-exp} holds. Hence, we can also assume that the gradients $\nabla u_n$ converge almost everywhere to $\nabla u$. From this, it  follows the convergence almost everywhere of $a(t,x,u_n,\nabla u_n)$ to $a(t,x,u,\nabla u)$ and of $T_n\pare{H(t,x,u_n,\N u_n)}$ to $H(t,x,u,\N u)$, which, together with \eqref{deboleL1}, implies \eqref{Conv:1.2-exp}. Finally, \eqref{Conv:1.45-exp} follows from   \eqref{Conv:1.3-exp} and the fact that $a(t,x,u_n,\nabla u_n)$ is bounded in $(L^2(Q_T))^N$. Taking into account the above convergences, we may follow the argument of   \cite[Proposition 6.4]{DGLS}, with minor modifications, to prove \eqref{Conv:1.3bis-exp}.
\end{proof}

Some  straightforward consequences of the pointwise convergence \eqref{Conv:1-exp} and the estimates \eqref{stimaprioriexp}, \eqref{altrastimapriori0} are
\begin{equation*}%
    \Psi(u)\in L^2(0,T; H_0^1(\Omega))\cap L^\infty(0,T;L^2(\Omega)),\qq
    H(t,x,u,\nabla u)\,\varphi(u)\in L^1(Q_T).%
\end{equation*}
As a consequence, $H(t,x,u,\nabla u)\,u\in L^1(Q_T)$.
Thus, on account of \eqref{ec:3.00}, condition \eqref{regolarexp} holds true.
Moreover it follows from Corollary \ref{Cor:1} that
\begin{equation*}%
u \in L^\infty(0,T; L^p(\Omega))\quad\mbox{for every $p<\infty$.}
\end{equation*}
Since, by \eqref{Conv:1.3bis-exp}, $u\in C([0,T]; L^1(\Omega))$, it follows easily using interpolation that
\begin{equation*}%
    u \in C([0,T]; L^p(\Omega))\quad\mbox{for every $p<\infty$.}
\end{equation*}

\subsection{End of the proof of Theorem \ref{teo:global-bdd}}
We wish to check the weak formulation in Definition \ref{defsol}. For $t\in[0,T]$, let $v(\tau,x) \in L^2(0,t; H^1_0(\Omega)) \cap L^\infty(Q_t)\cap C([0,t]; L^2(\Omega))$ be such that $\partial_\tau v\in L^2(0,t;H^{-1}(\Omega))$. Then
\begin{align*}
\intO u_n(t,x)v(t,x)\,dx-\intO T_n(u^0(x))v(0,x)\,dx&-\int_{0}^t\langle \partial_\tau v, u_n\rangle\,d\tau+\iint_{Q_t}a(\tau,x,u_n,\N u_n)\cdot\N v\,dx\,d\tau \\
&=\iint_{Q_t}T_n\pare{H(\tau,x,u_n,\N u_n)}v\,dx\,d\tau\,.
\end{align*}
Letting $n$ go to infinity in the above identity and using the convergences \eqref{Conv:1.3bis-exp}, \eqref{Conv:3-exp}, \eqref{Conv:1.45-exp} and \eqref{Conv:1.2-exp} we obtain the weak formulation \eqref{formdeb}.

Finally, we recover the energy formulation \eqref{energy} by taking $T_k(u_n)$ as test function in the approximating problem \eqref{eq:sol-pbn}. We can repeat similar computations as before to take, first, the limit on $n$ using \eqref{Conv:1.25-exp}, and then in $k$ thanks to the integrability properties of $u$.

\section{Proof of Theorem \ref{teo:small} - Existence with smallness conditions and unbounded data}\label{sec:small}

\subsection{A priori estimates}\label{sec:data-small}

In what follows, we will need to use the Gagliardo-Nirenberg interpolation inequality (see, for instance, \cite[Theorem 2.1]{DiBen}), that we recall below.

\begin{proposition}\label{teo:GN}
Let $\Omega\subset \mathbb{R}^N$ ($N\ge3$) be a bounded and open subset and $T$ a real positive number. If
\begin{equation*}\label{GN}
v\in L^{\infty}(0,T;L^{h}(\Omega))\cap L^{\eta}(0,T;W_0^{1,\eta}(\Omega)),\q\t{with}\q 1\le \eta<N \quad \text{and}\quad 1\le h\le\eta^*,
\end{equation*}
then
\begin{equation*}
v\in L^y(0,T;L^w(\Omega)),\q\t{with}\q \frac{Nh}{w}+\frac{ N(\eta-h)+\eta h }{y}=N,\q \eta\le y\le \infty,\q h\le w\le \eta^*.
\end{equation*}
Moreover, the following inequality holds:
\begin{equation*}\label{disGN}
\int_0^T \|v(t)\|_{L^w(\Omega)}^y\,dt\le c(N,\eta,h)\|v\|_{L^{\infty}(0,T;L^h(\Omega))}^{y-\eta}\int_0^T\|\N v(t)\|_{L^{\eta}(\Omega)}^{\eta}\,dt.
\end{equation*}
In particular, imposing $w=y$, this implies 
\begin{equation*}
v\in L^{\eta\frac{N+h}{N}}(Q_T).
\end{equation*}
\end{proposition}

\begin{proposition}[The case $\frac{4}{N+2}\le q<2$ and $\si=\frac{Nq}{2-q}\ge 2$]\label{prop:ape}
Assume that
\begin{equation*}
u^0\in L^\si(\Omega),
\end{equation*}
\begin{equation*}
 f\in L^r(0,T;L^m(\Omega))\q\t{with}\q\frac{N}{m}+\frac{2}{r}= 1+\frac{2}{q}\,,
\end{equation*}
where
\[1\le r\le \sigma\,,\q\frac{N^2q}{N(q+2)-2(2-q)}\le m\le \sigma\,,\]
and that their norms in those spaces are small enough (depending on $N,q,r,|\Omega|$, see an explicit bound in \eqref{smallness} below).
Then the  following a priori estimate holds:
\begin{equation}\label{stimaenergia}
\max_{t\in[0,T]}\intO |u_n(t)|^\si\,dx+\iint_{Q_T}\!|\nabla u_n|^2 |u_n|^{\sigma-2}\,dx\,dt
+\iint_{Q_T}\!|\nabla u_n|^2 \,dx\,dt\le C(N,q, r, |\Omega|) \quad\mbox{$\forall\,n\in \mathbb N$.}
\end{equation}
Moreover,  $\set{\av{\N u_n}^q |u_n|^\sigma}$ is uniformly bounded in $L^1(Q_T)$ and consequently so is $\set{\av{\N u_n}^q |u_n|}$.
Finally,
\begin{equation*}
\t{$\{u_n\}$ is bounded in $L^{\sigma\frac{N+2}{N}}(Q_T)$.}
\end{equation*}

\end{proposition}

\begin{proof}

\noindent
We use $|u_n|^{\sigma-1} {\rm sign}(u_n)$ as test function in \eqref{eq:pbn} so that, recalling \eqref{eq:A} and \eqref{H},  we obtain
\begin{align*}
&\frac1\sigma \intO |u_n(t)|^\sigma\,dx+(\si-1)\iint_{Q_t} \av{\N u_n }^2\av{u_n}^{\si-2}\,dx\,d\tau\\
&\le\iint_{Q_t} \av{\N u_n}^q |u_n|^\sigma\,dx\,d\tau+\iint_{Q_t} |f| |u_n|^{\sigma-1}\,dx\,d\tau +\frac1\sigma \intO |u^0|^\sigma\,dx.
\end{align*}
To estimate the integral with  $|\nabla u_n|^q$, we apply H\"older's inequality twice, first with exponents $\Big(\frac{2}{q},\frac{2}{2-q}\Big)$ and then with  exponents $\pare{\frac{2^*}{2},\frac{N}{2}}$, and finally we use Sobolev's inequality:
\begin{align*}
\intO \av{\N u_n}^q |u_n|^\sigma\,dx&=\intO \av{\N u_n}^q |u_n|^{\frac{q(\si-2)}{2}}|u_n|^{\sigma-\frac{q(\si-2)}{2}}\,dx\\
&\le \bigg(\intO \av{\N u_n}^2 \av{u_n}^{\si-2}\,dx\bigg)^\frac{q}{2}
\bigg(\intO |u_n|^{\si+\frac{2q}{2-q}}\,dx\bigg)^\frac{2-q}{2}\\
&\le \bigg(\intO \av{\N u_n}^2 \av{u_n}^{\si-2}\,dx\bigg)^\frac{q}{2}
\underbracket[0.14ex]{\bigg(\intO |u_n|^{\frac{\sigma2^*}{2}}\,dx\bigg)^\frac{2-q}{2^*}}_{\le c(\intO |\nabla u_n|^2|u_n|^{\si-2}\,dx)^\frac{2-q}{2}}
\bigg(\intO \underbracket[0.14ex]{|u_n|^{\frac{Nq}{2-q}}}_{=|u_n|^\si}\,dx\bigg)^\frac{2-q}{N}\\
&\le c(N,q) \bigg(\intO \av{\N u_n}^2 \av{u_n}^{\si-2}\,dx\bigg)
\bigg(\intO |u_n|^{\si}\,dx\bigg)^\frac{q}{\si}\,.
\end{align*}
Integrating from $0$ to $t$, it yields
\begin{equation}\label{eq:unifL^1}
\iint_{Q_t} \av{\N u_n}^q |u_n|^\sigma\,dx\, d\tau\le c\sup_{\tau\in[0,t]}\pare{\intO |u_n(\tau)|^{\si}\,dx}^{\frac{q}{\si}}\left(\iint_{Q_t}\av{\N u_n}^2 \av{u_n}^{\si-2}\,dx\, d\tau\right).
\end{equation}

We have to estimate the integral with the forcing term. Having in mind \eqref{eq:f-0b}, we use H\"older's inequality to obtain that
\begin{align}
\iint_{Q_t} |f| |u_n|^{\si-1}\,dx\,d\tau&\le \normaF{f}{L^r(0,T;L^m(\Omega))}\normaF{|u_n|^{\si-1}}{L^{r'}(0,t;L^{m'}(\Omega))}\nonumber\\
&= \normaF{f}{L^r(0,T;L^m(\Omega))}\normaF{u_n}{L^{r'(\si-1)}(0,t;L^{m'(\si-1)}(\Omega))}^{\si-1}.\label{eq:est-f}
\end{align}
To estimate the last integral, we use the Gagliardo-Nirenberg interpolation inequality  (Proposition \ref{teo:GN}) with
 $v_n=|u_n|^{\si/2}$, deducing
\begin{align}
\normaF{u_n}{L^{\frac{a\si}2}(0,t;L^{\frac{b\si}2}(\Omega))}^\frac{\si}2&=
\normaF{v_n}{L^a(0,t;L^b(\Omega))}\le c \normaF{v_n}{L^\infty(0,t;L^2(\Omega))}^{\frac{a-2}a}\bigg(\iint_{Q_t}|\N v_n|^2\,dx\,d\tau\bigg)^{\frac1a}\nonumber\\
&\le c\sup_{\tau\in[0,t]}\bigg(\int_\Omega |u_n(\tau)|^\si\,dx\bigg)^{\frac{a-2}{2a}}\bigg(\iint_{Q_t}|\N u_n|^2 |u_n|^{\si-2}\,dx\,d\tau\bigg)^\frac1a\label{gaglianir0}
\end{align}
for all $a,b$ such that
\begin{equation}\label{gaglianir}
   2 \le b \le 2^*, \quad 2\le a <\infty, \quad \frac{N}b + \frac2a =\frac{N}2\,.
\end{equation}
In our case we must take $a$ and $b$ defined by
$$
   r'(\si-1) = \frac{a\si}{2},\quad m'(\si-1) = \frac{b\si}{2}\,.
$$
With this choice, it is easy to check that \eqref{gaglianir} is equivalent to the equality in assumption \eqref{eq:f-0b}.
Then \eqref{eq:est-f} and \eqref{gaglianir0} imply
\begin{align*}
  &  \iint_{Q_t} |f| |u_n|^{\si-1}\,dx\,d\tau\\
&\le
    c\normaF{f}{L^r(0,T;L^m(\Omega))}\sup_{\tau\in[0,t]}\bigg(\int_\Omega |u_n(\tau)|^\si\,dx\bigg)^{\frac{1}{r}- \frac1\si}
    \bigg(\iint_{Q_t}|\N u_n|^2 |u_n|^{\si-2}\,dx\,d\tau\bigg)^\frac1{r'}\\
    &\le \int_0^T \normaF{f(\tau)}{L^m(\Omega)}^r\,d\tau
    + c \sup_{\tau\in[0,t]}\bigg(\int_\Omega |u_n(\tau)|^\si\,dx\bigg)^{\left(\frac{1}{r}- \frac1\si\right)r'} \iint_{Q_t}|\N u_n|^2 |u_n|^{\si-2}\,dx\,d\tau
\end{align*}
by Young's inequality.

We gather the previous estimates:
\begin{align}
\frac1\si &\intO |u_n(t)|^\si\,dx+(\si-1)\iint_{Q_t} \av{\N u_n }^2 |u_n|^{\si-2} \,dx\,d\tau\label{eq:fin-est}\\
&\le c_1 \bigg[\sup_{\tau\in[0,t]}\bigg(\int_\Omega |u_n(\tau)|^\si\,dx\bigg)^{\frac q\si}+\sup_{\tau\in[0,t]}\bigg(\int_\Omega |u_n(\tau)|^\si\,dx\bigg)^{\left(\frac{1}{r}- \frac1\si\right)r'}\bigg]\iint_{Q_t}|\N u_n|^2 |u_n|^{\si-2}\,dx\,d\tau\nonumber\\
&\q +\frac1\si\intO |u^0|^\si\,dx+\int_0^T \normaF{f(\tau)}{L^m(\Omega)}^r\,d\tau.\nonumber
\end{align}

We now set $\de=\de(N,q,r, |\Omega|)>0$ such that
\begin{equation*}
c_1\Big(\de^{\frac q\si}+\de^{\left(\frac{1}{r}- \frac1\si\right)r'}\Big)\le \frac{\si-1}{2}\,,
\end{equation*}
and we assume that the data $u^0$ and $f$ verify the smallness condition
\begin{equation}\label{smallness}
\frac1\si\intO |u^0|^\si\,dx+\int_0^T \normaF{f(\tau)}{L^m(\Omega)}^r\,d\tau \le \frac{\de}{2\si}\,.
\end{equation}
We define
\[
T^*_n=\sup\bigg\{t\in[0,T]:\q \int_\Omega |u_n(\tau)|^\si\,dx\le \de\q \forall \tau\in [0,t]\bigg\}.
\]
Note that $T_n^*>0$ by the smallness condition \eqref{smallness} and the continuity of the function $t\mapsto \intO |u_n(t)|^\si\,dx$.
Then, if $t\le T_n^*$, the estimate \eqref{eq:fin-est} reads
\begin{align}\label{eq:fin-est-2}
&\frac1\sigma\intO |u_n(t)|^\si\,dx+\frac{\si-1}{2}\iint_{Q_t} \av{\N u_n }^2\av{u_n}^{\si-2} \,dx\,d\tau \le \frac1\si\intO |u^0|^\si\,dx+\int_0^T \normaF{f(\tau)}{L^m(\Omega)}^r\,d\tau  \le \frac{\de}{2\si}\,.
\end{align}
We want to show that $T_n^*=T$. Let us suppose that $T_n^*<T$.
Then, if we take $t= T_n^*$ in \eqref{eq:fin-est-2}, we obtain a contradiction:
\begin{align*}
\frac{\de}\si=\frac1\si\intO |u_n(T_n^*)|^\si\,dx&\le \frac1\si\intO |u^0|^\si\,dx+\int_0^T \normaF{f(\tau)}{L^m(\Omega)}^r\,d\tau \le \frac{\de}{2\si}\,.
\end{align*}
Therefore $T_n^*=T$ for all $n$, and \eqref{eq:fin-est-2} provides an estimate for the first two integrals in
\eqref{stimaenergia}.\\
The previous computations allow us to say that
\[
\t{$\set{u_n}$ is bounded in $L^{\si\frac{N+2}{N}}(Q_T)$,}
\]
 thanks to an application of  Gagliardo-Nirenberg's inequality to   $|u_n|^{\si/2}$ in the spaces
$L^\infty(0,T;L^2(\Omega))\cap L^2(0,T;H_0^1(\Omega))$.

To estimate the last integral in \eqref{stimaenergia}, we take $u_n$ as test function in \eqref{eq:pbn}, obtaining
\begin{align*}
\frac{1}{2}\intO u_n^2(t)\,dx+\iint_{Q_t} \av{\N u_n }^2 \,dx\,d\tau
&\le\iint_{Q_t} \av{\N u_n}^q u_n^2\,dx \,d\tau+\iint_{Q_t} |f|\, |u_n|\,dx\,d\tau+\frac{1}{2}\intO (u^0)^2\,dx.
\end{align*}
Since
\begin{align*}
\iint_{Q_t}\av{\N u_n}^q u_n^2\,dx \,d\tau&\le \frac12\iint_{Q_t} \av{\N u_n}^2\,dx \,d\tau+c\iint_{Q_t} |u_n|^{\frac{4}{2-q}}\,dx \,d\tau
\end{align*}
by Young's inequality,
we just need to check that
\[
\frac{4}{2-q}\le \si\frac{N+2}{N},
\]
which is true in our range of $q$.
We now apply  H\"older's inequality to the term involving the source term, obtaining
\begin{align*}
\iint_{Q_t} \av{f} \av{u_n}\,dx\,d\tau\le \norm{f}_{L^{r}(0,T;L^{m}(\Omega))}\norm{u_n}_{L^{r'}(0,T;L^{m'}(\Omega))},
\end{align*}
which is bounded by \eqref{eq:f-0} and since  $m'\le m'(\si-1)$, $r'\le r'(\si-1)$ being $\si \ge2$ (see \eqref{eq:est-f}).

Finally, going back to \eqref{eq:unifL^1},  we obtain the desired $L^1$ estimate of $\{|\nabla u_n|^q|u_n|^\sigma\}$.
\end{proof}

\begin{remark}[On the level set function]\label{rmk:Gk}
\normalfont It is worth observing that we cannot deal with problem \eqref{eq:pb} as in \cite{M,MO}, i.e., we cannot avoid the smallness assumption on the norm of the data by using test functions which involve the level set function $G_k(u)$ in \eqref{tk1}. Indeed, if we take $G_k(u_n)^{\si-1}\t{sign}(u_n)$ as test function in \eqref{eq:pbn}, the estimate on the superlinear term becomes
\begin{align*}
&\iint_{Q_t}  |\nabla  u_n|^q\av{u_n}  \av{G_k(u_n)}^{\si-1} \,dx\,d\tau\\
&= \iint_{Q_t}  |\nabla  G_k(u_n)|^q\av{ G_k(u_n)}^{\si} \,dx\,d\tau+k\iint_{Q_t}  |\nabla  u_n|^q  \av{G_k(u_n)}^{\si-1} \,dx\,d\tau ,
\end{align*}
and we have no control on the last integral.
\end{remark}

\begin{proposition}[The case $\frac{2}{N+1}<q<\frac4{N+2}$ and $\si=\frac{Nq}{2-q}$]\label{prop:ape2}
Assume that \eqref{eq:A} and \eqref{H} are in force, and that the data are as in \eqref{eq:f-0}, i.e.
\begin{equation*}
u^0\in L^\si(\Omega),
\end{equation*}
\begin{equation*}
 f\in L^r(0,T;L^m(\Omega))\q\t{with}\q\frac{N}{m}+\frac{2}{r}= 1+\frac{2}{q}\,,
 \q 1\le r\le \sigma\,,
 \q\frac{N^2q}{N(q+2)-2(2-q)}\le m\le \sigma\,,
\end{equation*}
and that their norms in those spaces are small enough (depending on $N,q,r,|\Omega|$, see explicit bound \eqref{smallness0} below).
Then the  following a priori estimate holds:
\begin{equation*}\label{stimaenergia2}
\max_{t\in[0,T]}\intO |u_n(t)|^\si\,dx+\iint_{Q_T}\frac{|\nabla u_n|^2}{(1+|u_n|)^{2-\sigma}}\,dx\,d\tau
\le C(N,q, r, |\Omega|) \quad\mbox{$\forall\,n\in \mathbb N$.}
\end{equation*}
Moreover,  $\set{\av{\N u_n}^q |u_n|^\sigma}$ is uniformly bounded in $L^1(Q_T)$, wherewith so is $\set{\av{\N u_n}^q |u_n|}$. Finally,
\begin{equation}\label{cor-GN2}
\t{$\{u_n\}$ is bounded in $L^{\sigma\frac{N+2}{N}}(Q_T)$.}
\end{equation}
\end{proposition}

We remark that, since in this case $1<\sigma<2$,  it is not possible to take $|u_n|^{\sigma -1}{\rm sign}(u_n)$ as test function on the approximate problems. We will have to consider  a regularization of this test function. We therefore define the following functions, which depend on a positive parameter $\eps$:
\begin{subequations}\label{funzioni1}
\begin{align}
\varphi(s) &  =\big[(\eps+|s|)^{\si-1}-\eps^{\si-1}\big]{\rm sign}(s),\label{funzioni1a}\\
\Phi(s)&= \int_0^s
\varphi(\eta)\,d\eta = \frac1\si\big[(\eps+|s|)^\si-\eps^\si\big]-\eps^{\si-1}|s|,\qquad
\Psi(s)=(\eps+|s|)^{\si/2}-\eps^{\si/2}.\label{funzioni1b}
\end{align}
\end{subequations}

\begin{lemma}\label{lemmafunzioni} For $1<\sigma<2$ and $\eps>0$, let $\varphi$, $\Phi$ and $\Psi$ be the functions defined in  \eqref{funzioni1}. Then, for every $s\in\R$:
\begin{align}
\label{disug1}
\frac{2(\si-1)}{\si^2}\Psi(s)^2&\stackrel{\mbox{\rm (a)}}{\le}\Phi(s)\stackrel{\mbox{\rm (b)}}\le \Psi(s)^2;\\
\label{disug2}
|\varphi(s)|&\le\Psi(s)^{\frac{2(\si-1)}\si};\\
\label{disug3}
s\varphi(s)&\le \frac2\si \Psi(s)^{2}.
\end{align}

\end{lemma}

\begin{proof}
For symmetry, it is enough to prove the inequalities for $s\ge 0$.
\begin{itemize}
\item {\eqref{disug1}(a):} If we define
$$
   g(s) = \Phi(s) - \frac{2(\si-1)}{\si^2} \Psi(s)^2,
$$
then it is easy to check that $g(0)=g'(0)=0$, and that 
$$
   g''(s) = \frac{(\si-1)(2-\si)}{\si}(\eps+s)^\frac{\si-4}2\, \Psi(s) \ge 0,
$$
therefore $g(s)\ge 0$ for all $s\in \R$.

\item {\eqref{disug1}(b):} Very similar to the previous one.

\item {\eqref{disug2}:} We wish to check that
$$
  h(s):=\Psi(s)^{\frac{2(\si-1)}\si}  - \varphi(s) \ge 0.
$$
Then $h(0)=0$ and, for all $\s\ge0$,
$$
    h'(s)= (\si-1)(\eps+s)^{\frac{\sigma-2}{2}}\Big\{\big[(\eps+s)^{\si/2}-\eps^{\si/2}\big]^{\frac{\si-2}\si}- (\eps+s)^\frac{\si-2}{2}\Big\}\,.
$$
Then $h'(s)\ge0$ if and only if
$$
  \big[(\eps+s)^{\si/2}-\eps^{\si/2}\big]^{\frac{\si-2}\si}\ge (\eps+s)^\frac{\si-2}{2}\,,
$$
that is (since $\si<2$),
$$
   (\eps+s)^{\si/2}-\eps^{\si/2}\le (\eps+s)^{\si/2}\,,
$$
which is true for every $s\ge 0$.

\item {\eqref{disug3}:} In this case, if we define
$$
   k(s) = \frac2\si \Psi(s)^2- s \varphi(s),
$$
then
$$
   k(0)=k'(0)=0,\qquad k''(s) = (2-\si)\Big[\eps^{\sigma/2}(\eps+s)^\frac{\si-4}2 + (\si-1)s(\eps+s)^{\si-3}\Big]\ge 0 \quad \forall\,s\ge0.
$$
\end{itemize}
\end{proof}

\begin{proof}[Proof of Proposition \ref{prop:ape2}]
The proof follows the same scheme as that of Proposition \ref{prop:ape}, but with powers of $u_n$ replaced by the functions defined in \eqref{funzioni1}.
We take $\vp(u_n)$ as test function in \eqref{eq:pbn}, obtaining
\[
\intO \Phi(u_n(t))\,dx+(\si-1)\iint_{Q_t} \frac{\av{\N u_n }^2}{\pare{\eps+\av{u_n}}^{2-\si}}
\le\underbracket[0.14ex]{\iint_{Q_t} \av{\N u_n}^q u_n\vp(u_n)}_{\fbox{A}}+\underbracket[0.14ex]{\iint_{Q_t} |f|\, |\vp(u_n)|}_{\fbox{B}}+\intO \Phi(u^0)\,dx,
\]
thanks also to \eqref{eq:A} and \eqref{H}.\\
To estimate the integral \fbox{A}, we first apply inequality \eqref{disug3} and then H\"older's inequality with exponents $\pare{\frac{2}{q},\frac{2}{2-q}}$, obtaining:
\begin{align*}
\intO \av{\N u_n}^q u_n\vp(u_n)\,dx
&\le \frac2\si\intO \av{\N u_n}^q (\Psi(u_n))^{2}\,dx
= c\intO \frac{\av{\N u_n}^q}{ \pare{\eps+\av{u_n}}^{\frac{q(2-\si)}{2}}}(\Psi(u_n))^{2}\pare{\eps+\av{u_n}}^{\frac{q(2-\si)}{2}}\,dx\\
&\le c\pare{\intO \frac{\av{\N u_n}^2 }{\pare{\eps+\av{u_n}}^{2-\si}}\,dx}^\frac{q}{2}
\pare{\intO \Psi(u_n)^\frac{4}{2-q}\pare{\eps+\av{u_n}}^{\frac{q(2-\si)}{2-q}}\,dx}^\frac{2-q}{2}\,.
\end{align*}
We focus on the last integral in the inequality above: recalling the definition of $\si$ and using H\"older's inequality with exponents   $\pare{\frac{2^*}{2},\frac{N}{2}}$ and then Sobolev's inequality, we obtain
\begin{align*}
\bigg(\intO \Psi(u_n)^\frac{4}{2-q}(\eps+\av{u_n})^{\frac{q(2-\si)}{2-q}}\,dx\bigg)^\frac{2-q}{2}
&= \pare{\intO \Psi(u_n)^2\Psi(u_n)^\frac{2\si}{N}\pare{\eps+\av{u_n}}^{\frac{\si(2-\si)}{N}}\,dx}^\frac{2-q}{2}\\
&\le \pare{\intO  \Psi(u_n)^{2^*}\,dx}^\frac{2-q}{2^*}
\pare{\intO  \Psi(u_n)^\si \pare{\eps+\av{u_n}}^{\frac{\si(2-\si)}{2}}\,dx}^\frac{2-q}{N}\\
&\le  c\pare{\intO  \frac{|\nabla u_n|^2}{\pare{\eps+\av{u_n}}^{2-\si}}\,dx}^\frac{2-q}{2}
\pare{\intO  \Psi(u_n)^\si \pare{\eps+\av{u_n}}^{\frac{\si(2-\si)}{2}}\,dx}^\frac{q}{\si}.
\end{align*}
Moreover, since $(\eps+u_n)^{\si/2}= \Psi(u_n) +\eps^{\si/2}$,
\begin{align*}
&\bigg(\intO  \Psi(u_n)^\si (\eps+\av{u_n})^{\frac{\si(2-\si)}{2}}\,dx\bigg)^\frac{q}{\si}
= \bigg(\intO  \Psi(u_n)^\si (\Psi(u_n)+\eps^{\si/2})^{2-\si}\,dx\bigg)^\frac{q}{\si}\\
&\le c\bigg(\intO  \Psi(u_n)^\si \big(\Psi(u_n)^{2-\si}+\eps^{\frac{\si(2-\si)}2}\big)\,dx\bigg)^\frac{q}{\si}
\stackrel{\t{[Young]}}\le c\bigg(\intO  \big(\Psi(u_n)^2 +\eps^{\si}\big)\,dx\bigg)^\frac{q}{\si}\\
&\stackrel{\eqref{disug1}}\le
   c\bigg(\bigg[\intO  \Phi(u_n)\,dx\bigg]^\frac{q}{\si} +|\Omega|^{\frac{2-q}N}\eps^{q}\bigg)
   \le c\bigg(\bigg[\intO  \Phi(u_n)\,dx\bigg]^\frac{q}{\si} +\eps^{q}\bigg).
\end{align*}
Putting the last estimates together and integrating, we obtain
\begin{equation*}
\fbox{A} \le c \bigg(\iint_{Q_t} \frac{\av{\N u_n}^2}{ \pare{\eps+\av{u_n}}^{2-\si}}\,dx\,d\tau\bigg)\bigg(\sup_{\tau\in[0,t]}\bigg[\intO  \Phi(u_n(\tau))\,dx\bigg]^\frac{q}{\si} +\eps^{q}\bigg).
\end{equation*}
We only have to estimate the term \fbox{B}. The assumption on $f$ and inequalities \eqref{disug1} and \eqref{disug2} imply
\begin{align*}
\fbox{B}&\le \normaF{f}{L^r(0,t;L^m(\Omega))}\normaF{\vp(u_n)}{L^{r'}(0,t;L^{m'}(\Omega))} \le
\normaF{f}{L^r(0,T;L^m(\Omega))}\normaF{\Psi(u_n)}{L^{\frac{2r'}{\si'}}(0,t;L^{\frac{2m'}{\si'}}(\Omega))}^{\frac2{\si'}}
.
\end{align*}
Therefore, by  applying the Gagliardo-Nirenberg inequality to the function $\Psi(u_n)$ and proceeding as in the proof of Proposition \ref{prop:ape}, we obtain
$$
    \fbox{B}\le \int_0^T \normaF{f(\tau)}{L^m(\Omega)}^r\,d\tau
    + c \sup_{\tau\in[0,t]}\bigg(\int_\Omega \Phi(u_n(\tau))\,dx\bigg)^{\left(\frac{1}{r}- \frac1\si\right)r'} \iint_{Q_t}\frac{|\N u_n|^2}{(\eps+|u_n|)^{2-\si}}\,dx\,d\tau.
$$
Gathering the previous estimates, we arrive at
\begin{align}
&\intO \Phi(u_n(t))\,dx+(\si-1)\iint_{Q_t} \av{\N u_n}^2 (\eps+|u_n|)^{\si-2} \,dx\,d\tau\label{eq:fin-est0}\\
&\le c_1 \bra{\eps^q + \sup_{\tau\in[0,t]}\bigg(\int_\Omega \Phi(u_n(\tau))\,dx\bigg)^{\frac q\si}+\sup_{\tau\in[0,t]}\bigg(\int_\Omega \Phi(u_n(\tau))\,dx\bigg)^{\left(\frac{1}{r}- \frac1\si\right)r'}}\iint_{Q_t}\frac{|\N u_n|^2}{(\eps+|u_n|)^{2-\si}}\,dx\,d\tau\nonumber\\
&\q +\intO \Phi(u^0)\,dx+\int_0^T \normaF{f(\tau)}{L^m(\Omega)}^r\,d\tau.\nonumber
\end{align}
We now fix $\eps\in(0,1)$ and $\de>0$ such that
\begin{equation*}
c_1\eps^q<\frac{\si-1}{4}\,,\qquad c_1\Big(\de^{\frac q\si}+\de^{\left(\frac{1}{r}- \frac1\si\right)r'}\Big)\le \frac{\si-1}{4}\,,
\end{equation*}
and assume that the data $u^0$ and $f$ verify the smallness condition
\begin{equation}\label{smallness0}
\intO \Phi(u^0)\,dx+\int_0^T \normaF{f(\tau)}{L^m(\Omega)}^r\,d\tau \le \frac{\de}{2}\,.
\end{equation}
We define
$$
T_n^*=\sup\bigg\{t\in[0,T]:\q \int_\Omega \Phi(u_n(\tau))\,dx\le \de\q \forall \tau\in [0,t]\bigg\}.
$$
Note that $T_n^*>0$ by the smallness condition \eqref{smallness0} and the continuity of the function $t\mapsto \intO \Phi(u_n(t))\,dx$.
Then, if $t\le T_n^*$, the estimate \eqref{eq:fin-est0} reads
\begin{equation}\label{eq:fin-est-3}
\intO \Phi(u_n(t))\,dx+\frac{\si-1}{2}\iint_{Q_t} \frac{\av{\N u_n}^2}{(\eps +\av{u_n})^{2-\si}} \,dx\,d\tau\le \intO \Phi(u^0)\,dx+\int_0^T \normaF{f(\tau)}{L^m(\Omega)}^r\,d\tau \le \frac{\de}{2}\,.
\end{equation}
Suppose now that $T_n^*<T$.
Thus, if we take $t= T_n^*$ in \eqref{eq:fin-est-3}, we obtain a contradiction:
\begin{align*}
\de=\intO \Phi(u_n(T_n^*))\,dx&\le \intO \Phi(u^0)\,dx+\int_0^T \normaF{f(\tau)}{L^m(\Omega)}^r\,d\tau \le \frac{\de}{2}\,.
\end{align*}
Therefore $T_n^*=T$ for all $n$, and \eqref{eq:fin-est-3} provides the desired estimate (recall that $0<\eps<1$).

As before, the Gagliardo-Nirenberg inequality implies  estimate  \eqref{cor-GN2}.

Note that, going back to estimate of integral $\fbox{A}$, we have shown that $\{\av{\N u_n}^q |u_n|\, |\vp(u_n)|\}$ is uniformly bounded in $L^1(Q_T)$. On account of the growth of $\vp$, we are done.
\end{proof}

\begin{proposition}[The case $q=\frac{2}{N+1}$]\label{prop:ape-om}
If the integrals
\begin{equation}\label{small4}
\int_\Omega |u^0|(\log^*|u^0|)^{\frac{N+1}2}dx \quad\mbox{and}\quad \iint_{Q_T} |f|\,(\log^*|f|)^{\frac{N+1}2}\,dx\,d\tau
\end{equation}
are small enough (depending on $N$), then the  following a priori estimate holds:
\begin{equation}\label{stimaenergia-critico}
\max_{t\in[0,T]}\intO |u(t)|(\log^*{|u(t)|})^{\frac{N+1}2}\,dx+\iint_{Q_T}   \frac{(\log(1+\av{u_n}))^{\frac{N-1}2}}{1+\av{u_n}}|\nabla u_n|^2 \,dx\,dt\le C \q\forall n\in\mathbb{N}
\end{equation}
where $C$ depends on $N$, $|\Omega|$ and on the value of the integrals in \eqref{small4}.

Moreover,  the sequence $\big\{\av{\N u_n}^q |u_n|\log\pare{1+|u_n|}^{\frac{N+1}2}\}$ is uniformly bounded in $L^1(Q_T)$, wherewith so is $\set{\av{\N u_n}^q |u_n|}$.
\end{proposition}

\begin{proof}
We recall \eqref{eq:A} and \eqref{H} and we use $\pare{\log\pare{1+\av{u_n}}}^{\frac{N+1}{2}}\t{sign}(u_n)=\pare{\log\pare{1+\av{u_n}}}^{\frac{1}{q}}\t{sign}(u_n)$ as test function in \eqref{eq:pbn}, obtaining
\begin{align}
  &\int_\Omega \Phi(u_n(t))\,dx + \frac{N+1}2 \iint_{Q_t}\frac{|\nabla u_n|^2(\log(1+\av{u_n}))^{\frac{N-1}2}}{1+\av{u_n}}\,dx\,d\tau\nonumber\\
   &\le \iint_{Q_t}|\nabla u_n|^q \av{u_n}  (\log(1+\av{u_n}))^{\frac{N+1}2} \,dx\,d\tau+\iint_{Q_t}\av{f} (\log(1+\av{u_n}))^{\frac{N+1}2} \,dx\,d\tau + \int_\Omega \Phi(u^0)\,dx,\label{eq:main-q-critico}
\end{align}
where
\begin{equation}\label{Phi4}
\Phi(s)=\int_{0}^{|s|}\big(\log\pare{1+z}\big)^{\frac{N+1}{2}}\,dz.
\end{equation}
We define
\begin{equation}\label{phi4}
\Psi(s)=\int_{0}^{\av{s}}\frac{(\log(1+z))^{\frac{N-1}4}}{\sqrt{1+z}}\,dz,
\end{equation}
so that
$$
\iint_{Q_t}   \frac{|\nabla u_n|^2(\log(1+\av{u_n}))^{\frac{N-1}2}}{1+\av{u_n}} \,dx\,d\tau=\iint_{Q_t}  |\nabla u_n|^2(\Psi'(u_n))^2 \,dx\,d\tau
=\iint_{Q_t}  |\nabla \Psi(u_n)|^2\,dx\,d\tau.
$$
We now focus on the integral involving the $q$ power of the gradient. We first rewrite it as
\begin{align*}
&\iint_{Q_t}|\nabla u_n|^q \av{u_n}  (\log(1+\av{u_n}))^{\frac{N+1}2}\,dx\,d\tau\\
&=\iint_{Q_t}|\nabla u_n|^q (\Psi'(u_n))^q\av{u_n}  (\log(1+\av{u_n}))^{\frac{N+1}2}(\Psi'(u_n))^{-q}\,dx\,d\tau\\
&=\iint_{Q_t}|\nabla u_n|^q (\Psi'(u_n))^q\av{u_n}  (1+\av{u_n})^{\frac{1}{N+1}}\pare{\log(1+\av{u_n})}^{\frac{N^2+N+2}{2(N+1)}}\,dx\,d\tau.
\end{align*}
We use Hölder's inequality twice, first with exponents $(\frac 2q, \frac2{2-q})=(N+1,\frac{N+1}{N})$   and then with exponents $(\frac{ 2^*}{2}, \frac {N}{2})$, and also Sobolev's inequality,
 obtaining
\begin{align*}
&\iint_{Q_t}|\nabla u_n|^q \av{u_n}  (\log(1+\av{u_n}))^{\frac{N+1}2} \,dx\,d\tau\nonumber\\
&\le\int_0^t \pare{\intO |\nabla u_n|^2 (\Psi'(u_n))^2\,dx}^{\frac{1}{N+1}} \pare{\intO \av{u_n} ^{\frac{N+1}{N}} (1+\av{u_n})^{\frac{1}{N}}\pare{\log(1+\av{u_n})}^{\frac{N^2+N+2}{2N}} \,dx}^{\frac{N}{N+1}}\,d\tau \nonumber\\
&\le\int_0^t \pare{\intO |\nabla \Psi(u_n)|^2\,dx}^{\frac{1}{N+1}}
\pare{\intO( \Psi(u_n) )^{2^*}\,dx}^{\frac{N}{N+1}\frac{2}{2^*}} \pare{\intO \psi(u_n) \,dx}^{\frac{2}{N+1}}\,d\tau\nonumber \\
&\le c_1 \sup_{\tau\in[0,t]}  \pare{\intO \phi(u_n(\tau)) \,dx}^{\frac{2}{N+1}}  \iint_{Q_t} |\nabla \Psi(u_n)|^2\,dx
\,d\tau ,\label{eq:grad-critico}
\end{align*}
where $c_1=c_1(N)$ and
\begin{equation}\label{psi4}
\phi(s)=\av{s} ^{\frac{N+1}{2}} (1+\av{s})^{\frac{1}{2}}\pare{\log(1+\av{s})}^{\frac{N^2+N+2}{4}}( \Psi(s) )^{-N}.
\end{equation}
We will show that there exists a positive constant $c_2=c_2(N)$ such that
\begin{equation}\label{PsiPhi}
\phi(s) \le c_2\, \Phi(s)\quad \forall\,s\in \R.
\end{equation}
Indeed, using the definitions \eqref{Phi4}, \eqref{phi4} and \eqref{psi4} and de l'H\^opital's rule, we obtain
\begin{equation}\label{funzinfty}
\Phi(s)\sim|s|\big(\log(|s|)\big)^{\frac{N+1}{2}}\!,
\quad
\Psi(s)\sim 2\sqrt{|s|}\big(\log(|s|)\big)^{\frac{N-1}{4}}\!,
\quad
\phi(s)\sim c\,|s|\big(\log(|s|)\big)^{\frac{N+1}{2}}
\quad
\mbox{for $s\to \infty$,}
\end{equation}
\begin{equation}\label{funzzero}
\Phi(s)\sim c\, |s|^{\frac{N+3}2},
\quad
\Psi(s)\sim c\, |s|^{\frac{N-1}4},
\quad
\phi(s)\sim c\, |s|^{N+1}
\quad
\mbox{for $s\to 0$,}
\end{equation}
for some (different) constants $c$ depending on $N$.
From these formulas, it easily follows that the ratio $\phi(s)/\Phi(s)$ is bounded near both zero and $\infty$, therefore \eqref{PsiPhi} is proven.
Thus
\begin{equation}\label{stimagrad4}
\iint_{Q_t}|\nabla u_n|^q \av{u_n}  (\log(1+\av{u_n}))^{\frac{N+1}2} \,dx\,d\tau
\le c_3 \sup_{\tau\in[0,t]}  \pare{\intO \Phi(u_n(\tau)) \,dx}^{\frac{2}{N+1}}  \iint_{Q_t} |\nabla \Psi(u_n)|^2\,dx\,d\tau,
\end{equation}
where $c_3=c_3(N)$.\\
In order to estimate the integral involving the forcing term $f$,
we consider an increasing, convex function $\Gamma(s) : [0,+\infty)\to [0,+\infty)$ such that
\begin{equation}\label{Gamma4}
\Gamma(0)=0\,,\qquad \Gamma(s) = s^{\frac{N^2+N+2}{N(N+1)}}\exp \Big(\frac{N+2}Ns^{\frac2{N+1}}\Big)\quad \forall\,s\ge k_0=k_0(N)>0,
\end{equation}
where $\exp(z)=e^z$. Note that we cannot take $k_0=0$ because the function in \eqref{Gamma4}  is not convex near zero.
Then we can write
\begin{align*}
\iint_{Q_t}&\av{f} (\log(1+\av{u_n}))^{\frac{N+1}2} \,dx\,d\tau\nonumber\\
&=\iint_{Q_t\cap\set{\av{u_n}\le k_0}}\av{f} (\log(1+\av{u_n}))^{\frac{N+1}2} \,dx\,d\tau+\iint_{Q_t\cap \set{\av{u_n}> k_0}}\av{f} (\log(1+\av{u_n}))^{\frac{N+1}2} \,dx\,d\tau\nonumber\\
&\le (\log(1+k_0))^{\frac{N+1}2}\norm{f}_{L^1(Q_T)}
+\underbracket[0.14ex]{\iint_{Q_t\cap \set{\av{u_n}> k_0}}\av{f} (\log(1+\av{u_n}))^{\frac{N+1}2} \,dx\,d\tau}_{\fbox{I}}.\label{eq:fI}
\end{align*}
As far as the integral \fbox{I} is concerned, we apply Young's inequality with the function $\Gamma$ defined in \eqref{Gamma4} and its conjugate $\Gamma^*$, obtaining
\begin{equation}\label{stimaI}
\fbox{I}\le \iint_{Q_t}\Gamma^*(f)\,dx\,d\tau+\iint_{Q_t\cap \set{\av{u_n}> k_0}}\Gamma\pare{(\log(1+\av{u_n}))^{\frac{N+1}2}}\,dx\,d\tau,
\end{equation}
where 
$$\Gamma^*(s) = \int_0^s (\Gamma')^{-1}(\tau)\,d\tau.$$
Our choice of $\Gamma$, together with the estimates \eqref{funzinfty}, yield
\begin{align*}%
&\iint_{Q_t\cap \set{\av{u_n}> k_0}}\Gamma\pare{(\log(1+\av{u_n}))^{\frac{N+1}{2}}}\,dx\,d\tau\\
&=  \iint_{Q_t\cap \set{\av{u_n}> k_0}}(\log(1+\av{u_n}))^{\frac{N^2+N+2}{2N}}(1+\av{u_n})^{\frac{N+2}{N}}\,dx\,d\tau\\
&\le c_4\iint_{Q_t} \Phi(u_n)^{\frac2N}\,\Psi(u_n)^2\,dx\,d\tau\le c_4 \int_0^t \bigg(\int_\Omega\Phi(u_n(\tau))\,dx\bigg)^\frac2N\bigg(\int_\Omega\Psi(u_n(\tau))^{2^*}\,dx\bigg)^\frac2{2^*}\,d\tau \\
&\le c_5 \sup_{\tau\in [0,t]}\bigg(\int_\Omega\Phi(u_n(\tau))\,dx\bigg)^\frac2N
\iint_{Q_t} |\nabla\Psi(u_n)|^{2}\,dx\,d\tau,
\end{align*}
where $c_5=c_5(N)$. In the last passages we have used H\"older's and Sobolev's inequalities.
\\
To estimate the first integral in \eqref{stimaI}, we will show that the conjugate function $\Gamma^*(s)$ of $\Gamma(s)$ 
satisfies
\begin{equation}\label{Gamma*}
\Gamma^*(s) \leq c\,s\,(\log^* s)^\frac{N+1}2\quad \mbox{for all $s\ge 0$,}
\end{equation}
for some $c\in \R$. It is easy to check that
$$
   \Gamma'(t) \sim c\,t^\frac2N \exp \Big(\frac{N+2}N t^{\frac2{N+1}}\Big) \quad \mbox{for $t\to+\infty$.}
$$
From this it follows easily, using de l'H\^{o}pital's rule, that
\begin{align*}
   \lim_{s\to+\infty}\frac{\Gamma^*(s)}{s\,(\log s)^\frac{N+1}2} &=\lim_{s\to+\infty}\frac{(\Gamma')^{-1}(s)}{(\log s)^\frac{N+1}2}
   =
   \lim_{t\to+\infty}\frac{t}{(\log(\Gamma'(t)))^\frac{N+1}2} \\
   &=
   \lim_{t\to+\infty}\frac{t}{\Big(\log\Big(c\,t^\frac2N \exp \Big(\frac{N+2}Nt^{\frac2{N+1}}\Big)\Big)^\frac{N+1}2} = \pare{\frac{N}{N+2}}^{\frac{N+1}{2}}
   \in (0,+\infty)\,,
\end{align*}
which proves that the inequality in \eqref{Gamma*} is true for $s\ge 1$. Since $\Gamma^*$ is $C^1$ and $\Gamma^*(0)=0$, the inequality also holds for $s\in [0,1]$, if necessary with a different constant $c$. Therefore \eqref{Gamma*} is proved.
We can now estimate\eqref{stimaI}   as
$$
   \fbox{I}\le c_6\iint_{Q_T} |f|\,(\log^*|f|)^{\frac{N+1}2}\,dx\,d\tau  +
   c_5 \sup_{\tau\in [0,t]}\bigg(\int_\Omega\Phi(u_n(\tau))\,dx\bigg)^\frac2N
\iint_{Q_t} |\nabla\Psi(u_n)|^{2}\,dx\,d\tau\,.
$$
We gather all the previous computations for the terms in \eqref{eq:main-q-critico}, obtaining
\begin{align}
  &\int_\Omega \Phi(u_n(t))\,dx   \nonumber\\
  &\quad + 
  \bigg[\frac{N+1}2 - c_3 \sup_{\tau\in[0,t]}  \pare{\intO \Phi(u_n(\tau)) \,dx}^{\frac{2}{N+1}}
  - c_5 \sup_{\tau\in[0,t]}  \pare{\intO \Phi(u_n(\tau)) \,dx}^{\frac{2}{N}}\bigg] \iint_{Q_t}|\nabla \Psi(u_n)|^2\,dx\,d\tau
  \nonumber\\
   &\le   c_6 \iint_{Q_T} |f|\,(\log^*|f|)^{\frac{N+1}2}\,dx\,d\tau + c_7\norm{f}_{L^1(Q_T)} +\int_\Omega\Phi(u^0)\,dx \nonumber\\
&\le   c_8 \iint_{Q_T} |f|\,(\log^*|f|)^{\frac{N+1}2}\,dx\,d\tau +\int_\Omega\Phi(u^0)\,dx.\label{eq:main-critico}
\end{align}
We now fix $\de=\de(N)>0$ such that
\begin{equation*}
c_3\,\de^{\frac{2}{N+1}}+c_5\,\de^\frac{2}{N} \le \frac{N+1}4,
\end{equation*}
and assume that the data $u^0$ and $f$ verify the smallness condition
\begin{equation}\label{smallness-critico}
c_8 \iint_{Q_T} |f|\,(\log^*|f|)^{\frac{N+1}2}\,dx\,d\tau  +\int_\Omega\Phi(u^0)\,dx \le\frac\delta2\,.
\end{equation}
We define
$$
T_n^*=\sup\bigg\{t\in[0,T]:\q \int_\Omega  \Phi(u_n(\tau))\,dx \le \de\q \forall \tau\in [0,t]\bigg\}.
$$
Note that $T_n^*>0$ by the smallness condition \eqref{smallness-critico} and the continuity of the function $t\mapsto \intO \Phi(u_n(t))\,dx$.
Then, if $t\le {T_n}^*$, the estimate \eqref{eq:main-critico} reads
\begin{equation} \label{stima4}
  \int_\Omega \Phi(u_n(t))\,dx+ \frac{N+1}{4} \iint_{Q_t}|\nabla \Psi(u_n)|^2\,dx\,d\tau \le\frac\delta2.
\end{equation}
We can reason by contradiction as in the previous results to prove that $T_n^*=T$ and obtain the stated result.
Note that from \eqref{funzinfty} and \eqref{funzzero} it follows that
$$
   c_9 |s| (\log^*|s|)^\frac{N+1}2 - c_{10} \le \Phi(s) \le c_{11} |s| (\log^*|s|)^\frac{N+1}2,
$$
therefore condition \eqref{small4} implies \eqref{smallness-critico}, while the estimate \eqref{stima4} implies \eqref{stimaenergia-critico}. Finally, the estimate of $\set{\av{\N u_n}^q |u_n|\log\big(1+|u_n|\big)^{\frac{N+1}{2}}}$ in $L^1(Q_T)$ is a consequence of \eqref{stimagrad4}.
\end{proof}

In this case, we may also appeal to Gagliardo-Nirenberg's inequality to get an a priori estimate of the approximate solutions in a suitable Lebesgue space.
Even though this estimate is not as good as the one obtained in Proposition \ref{prop:ape-om}, its application in Proposition \ref{prop:cmpt-om} below is easier.
\begin{corollary}\label{cor-GN3}
Under the same assumptions as in Proposition \ref{prop:ape-om}, the following estimates hold
\begin{itemize}
\item $\ds \max_{t\in[0,T]}\int_\Omega|u_n(t)|\, dx+\iint_{Q_T}\frac{|\nabla u_n|^2}{1+|u_n|}\le C(N,|\Omega|)\qquad \forall n\in\bN\,$.
\item  $\{u_n\}$ is bounded in $L^{\frac{N+2}{N}}(Q_T)$.
\end{itemize}
\end{corollary}
\begin{proof}
Note that Proposition \ref{prop:ape-om} implies the first estimate.
Hence, the sequence  $\big\{\sqrt{1+|u_n|}-1\big\}$ is bounded in $L^\infty(0,T;L^2(\Omega))\cap L^2(0,T;H_0^1(\Omega))$ and so the result follows from the  Gagliardo-Nirenberg inequality.
\end{proof}

\begin{proposition}[The case $0<q<\frac{2}{N+1}$ and $\si=1$]\label{prop:ape-1}
Assume that
\[
u^0\in L^1(\Omega),\qq f\in L^1(Q_T)\,,
\]
and that their norms in those spaces are small enough (depending on $N,q,r,|\Omega|$, see explicit bound \eqref{eq:smallness-1} below).
Then the  following a priori estimate holds:
\begin{equation*}\label{stima3}
\max_{t\in [0,T]}\intO |u_n(t)|\,dx+\iint_{Q_T} \frac{\av{\N u_n }^2}{\pare{1+\av{u_n}}^{\lm+1}} \,dx\,d\tau\le C,\qquad
\mbox{with $\lm=\dfrac{2-q(N+1)}{N}\in(0,1)$.}
\end{equation*}
Moreover, $\set{\av{\N u_n}^q |u_n|}$ is uniformly bounded in $L^1(Q_T)$.
\end{proposition}

For $\eps>0$, we define the functions
\begin{subequations}\label{funzioni2}
\begin{align}
\varphi_\eps(s) &=\bigg(\frac{1}{\eps^\lm}-\frac{1}{\pare{\eps+|s|}^\lm}\bigg)\t{sign}{(s)}, \qquad \Psi_\eps(s)= \big[(\eps+|s|)^\frac{1-\lambda}2 - \eps^\frac{1-\lambda}2 \big]\\
\Phi_\eps(s)&= \int_0^s
\varphi_\eps(\eta)\,d\eta = \frac{|s|}{\eps^\lambda} - \frac1{1-\lambda} \big[(\eps+|s|)^{1-\lambda} - \eps^{1-\lambda} \big].
\end{align}
\end{subequations}

\begin{lemma}\label{lemmafunzioni2} Fix $\eps>0$ and let $\varphi_\eps$, $\Psi_\eps$ and $\Phi_\eps$ be the functions defined in  \eqref{funzioni2}. Then there exist positive constants $c_i$ (depending only on $\lambda$) such that, for every $\eps>0$ and for every $s\in\R$:
\begin{equation}\label{disug11}
    c_1\,\frac{|s|}{\eps^\lambda(\eps+|s|)} \le |\varphi_\eps(s)| \le c_2\,\frac{|s|}{\eps^\lambda(\eps+|s|)}\,;
\end{equation}
\begin{equation}\label{disug13}
    c_5\,\frac{|s|}{(\eps+|s|)^\frac{1+\lambda}2} \le \Psi_\eps(s) \le c_6\,\frac{|s|}{(\eps+|s|)^\frac{1+\lambda}2}\,;
\end{equation}
\begin{equation}\label{disug12}
    c_3\,\frac{s^2}{\eps^\lambda(\eps+|s|)} \le \Phi_\eps(s) \le c_4\,\frac{s^2}{\eps^\lambda(\eps+|s|)}\,.
\end{equation}
\end{lemma}

 \begin{proof}
We only prove \eqref{disug11}, the other two can be proved similarly.

The proof for $\eps=1$ follows from checking that the limits of $\dfrac{ |\varphi_1(s)|}{\av{s}/(1+|s|)}$ for $s\to0$ and for $|s|\to+\infty$ are finite and positive. The general case follows from the equality
$$
   \varphi_\eps(s) = \frac1{\eps^\lambda}\, \varphi_1\Big(\frac s\eps\Big).
$$
 \end{proof}

 \begin{proof}[Proof or Proposition \ref{prop:ape-1}]
In the following, we use the functions defined in \eqref{funzioni2}.\\
We recall \eqref{eq:A} and \eqref{H} and we take $\varphi_\eps(u)$ as test function in \eqref{eq:pbn}, obtaining
\begin{align*}
&\intO \Phi_\eps(u_n(t))\,dx+\lm\iint_{Q_t} \frac{\av{\N u_n }^2}{\pare{\eps+\av{u_n}}^{1+\lm}} \,dx\,d\tau\nonumber\\
&\le \iint_{Q_t} \av{\N u_n}^q |u_n||\vp_\eps(u_n)|\,dx \,d\tau+\iint_{Q_t} |f|\, |\vp_\eps(u_n)|\,dx\,d\tau+\intO \Phi_\eps(u^0)\,dx.
\end{align*}
Using inequality \eqref{disug11}, we deduce
\begin{align*}
&\intO \Phi_\eps(u_n(t))	\,dx+\lm\iint_{Q_t} \frac{\av{\N u_n }^2}{\pare{\eps+\av{u_n}}^{1+\lm}} \,dx\,d\tau\nonumber\\
&\le \frac{c}{\eps^\lm}\,\iint_{Q_t} \frac{\av{\N u_n}^q |u_n|^2}{\eps+|u_n|}\,dx \,d\tau+\frac{c}{\eps^\lm}\iint_{Q_t}\av{ f} \,dx\,d\tau+\intO \Phi_\eps(u^0)\,dx\nonumber.
\end{align*}
We deal with the term with $|\nabla u_n|^q$ using H\"older's inequality with exponents $\Big(\frac2q\,,\,\frac2{2-q}\Big)$, and obtaining
\begin{equation}\label{inter1}
\frac{1}{\eps^\lm}\,\int_{\Omega} \frac{\av{\N u_n}^q |u_n|^2}{\eps+|u_n|}\,dx \le \frac{1}{\eps^\lm}\bigg(\intO \frac{\av{\N u_n }^2}{\pare{\eps+\av{u_n}}^{1+\lm}} \,dx\bigg)^\frac{q}{2}
\bigg(\intO \pare{\eps+\av{u_n}}^\frac{(1+\lm)q-2}{2-q} |u_n|^\frac{4}{2-q}\,dx\bigg)^\frac{2-q}{2}.
\end{equation}
To estimate the last integral, we will use inequality \eqref{disug13}, then H\"older's inequality with exponents $\Big(\frac{2^*}2\,,\,\frac{N}2\Big)$, and finally Sobolev's inequality:
\begin{align*}
\bigg(\intO \pare{\eps+\av{u_n}}^\frac{(1+\lm)q-2}{2-q} |u_n|^\frac{4}{2-q}\,dx\bigg)^\frac{2-q}{2}
&\le c\,\bigg(\int_{\Omega}\Psi_\eps(u_n)^2 (\eps+|u_n|)^\frac{2\lambda}{2-q} |u_n|^\frac{2q}{2-q} \,dx\bigg)^\frac{2-q}{2}\\
&\le c\,\bigg(\int_{\Omega}\Psi_\eps(u_n)^{2^*}\,dx\bigg)^\frac{2-q}{2^*}
\bigg(\int_{\Omega} (\eps+|u_n|)^\frac{N\lambda}{2-q} |u_n|^\frac{Nq}{2-q} \,dx\bigg)^\frac{2-q}{N}\\
&\le c\,\bigg(\int_{\Omega}\frac{\av{\N u_n }^2}{\pare{\eps+\av{u_n}}^{1+\lm}}\,dx\bigg)^\frac{2-q}{2}
\bigg(\int_{\Omega} (\eps+|u_n|)^\frac{N\lambda}{2-q} |u_n|^\frac{Nq}{2-q} \,dx\bigg)^\frac{2-q}{N}.
\end{align*}
Thus, recalling the value of $\lm$, it follows from \eqref{inter1} that
$$
\frac{1}{\eps^\lm}\,\int_{\Omega} \frac{\av{\N u_n}^q |u_n|^2}{\eps+|u_n|}\,dx
\le \frac{c}{\eps^{\lm(1-\frac{2-q}N)}}\bigg(\int_{\Omega}\frac{\av{\N u_n }^2}{\pare{\eps+\av{u_n}}^{1+\lm}}\,dx\bigg)
\bigg(\int_{\Omega} \frac{1}{\eps^\lm}\,(\eps+|u_n|)^{1-\frac{Nq}{2-q}} |u_n|^\frac{Nq}{2-q} \,dx\bigg)^\frac{2-q}{N}.
$$
Now, it is easy to prove from \eqref{disug12} (see also Remark \ref{noteasyatall}) that
\begin{equation}\label{eq:itiseasytoprove}
    \frac{1}{\eps^\lm}\,(\eps+|u_n|)^{1-\frac{Nq}{2-q}} |u_n|^\frac{Nq}{2-q} \le c\pare{ \frac{1}{\eps^\lm}\frac{\av{u_n}^2}{\eps+\av{u_n}}+\eps^{1-\lm}}\le c(\Phi_\eps(u_n)+\eps^{1-\lambda})\,,
\end{equation}
therefore
\begin{align*}
\frac{1}{\eps^\lm}\,\int_{\Omega} \frac{\av{\N u_n}^q |u_n|^2}{\eps+|u_n|}\,dx
&\le \frac{c}{\eps^{\lm(1-\frac{2-q}N)}}\bigg(\int_{\Omega}\frac{\av{\N u_n }^2}{\pare{\eps+\av{u_n}}^{1+\lm}}\,dx\bigg)
\bigg(\int_{\Omega} (\Phi_\eps(u_n)+\eps^{1-\lambda}) \,dx\bigg)^\frac{2-q}{N}\\
&\le \frac{c}{\eps^{\lm(1-\frac{2-q}N)}}\bigg(\int_{\Omega}\frac{\av{\N u_n }^2}{\pare{\eps+\av{u_n}}^{1+\lm}}\,dx\bigg)
\bigg[\bigg(\int_{\Omega} \Phi_\eps(u_n)\,dx\bigg)^\frac{2-q}{N} +|\Omega|^{\frac{2-q}{N}} \eps^\frac{(1-\lambda)(2-q)}{N}\bigg]\,.
\end{align*}
Integrating, we have proved that
$$
\frac1{\eps^\lambda}\iint_{Q_t} \frac{\av{\N u_n}^q |u_n|^2}{\eps+|u_n|}\,dx\,d\tau
\le c \bigg(\int_{\Omega}\frac{\av{\N u_n }^2}{\pare{\eps+\av{u_n}}^{1+\lm}}\,dx\bigg)
\bigg[\frac{1}{\eps^\alpha}\sup_{\tau\in[0,t]}\bigg(\int_{\Omega} \Phi_\eps(u_n(\tau))\,dx\bigg)^\frac{2-q}{N}+\eps^{\beta}\bigg],
$$
where $\alpha= \lm\Big(1-\frac{2-q}N\Big)>0$ and $\beta=\frac{(1-\lambda)(2-q)}{N}>0$.

We summarize the previous computations:
\begin{align*}
    &\int_\Omega \Phi_\eps(u_n(t))\,dx + \bigg[ \lambda - \frac{C_1}{\eps^\alpha}\sup_{\tau\in[0,t]}\bigg(\int_{\Omega} \Phi_\eps(u_n(\tau))\,dx\bigg)^\frac{2-q}{N}-C_1\,\eps^{\beta}\bigg] \iint_{Q_t} \frac{\av{\N u_n }^2}{\pare{\eps+\av{u_n}}^{1+\lm}} \,dx\,d\tau\\
  & \le \frac{C_2}{\eps^\lambda} \iint_{Q_t}\av{ f} \,dx\,d\tau+\intO \Phi_\eps(u^0)\,dx.
\end{align*}
We first fix $\eps\in(0,1)$ such that $C_1\eps^\beta<\frac\lambda4$ and then $\delta>0$  such that
$$
    \frac{C_1}{\eps^\alpha}\, \delta^{\frac{2-q}{N}}< \frac\lambda4\,.
$$
Finally we assume that the data $f$ and $u^0$ are small enough:
\begin{equation}\label{eq:smallness-1}
    \frac{C_2}{\eps^\lambda} \iint_{Q_T}\av{ f} \,dx\,d\tau+\intO \Phi_\eps(u^0)\,dx\le\frac\delta2\,.
\end{equation}
This smallness condition allows us to conclude the proof reasoning as in the previous a priori estimates results.
\end{proof}

\begin{remark}\label{noteasyatall}
\normalfont We here provide the reader with some arguments aimed at proving \eqref{eq:itiseasytoprove}, i.e.
\begin{equation*}
(\eps+s)^{1-\frac{Nq}{2-q}} s^\frac{Nq}{2-q} \le c\pare{ \frac{s^2}{\eps+s}+\eps}\,\qq\forall s\ge0.
\end{equation*}
This inequality is easy to prove when $\eps=1$, indeed, the functions on the left and on the right hand side have the same behavior as $s\to+\infty$, while it is obvious when $s=0$. \\
The proof of the general case with $\eps\ne1$ follows by a homogeneity argument.
\end{remark}

\subsection{The case $\frac{4}{N+2}\le q< 2$: compactness and end of the proof} \label{ssec:cmptL2}

In this Section we assume that $ \frac4{N+2} \le q <2$ and that $\si= \frac{Nq}{2-q} \ge 2$. Starting from the estimates found in Proposition \ref{prop:ape}, we can prove the following compactness result:

\begin{proposition}[Compactness of the approximating sequence]\label{prop:cmpt}
The sequence of the approximate solutions \eqref{eq:pbn} satisfies, up to subsequences,
\begin{align}
u_n&\to u &&\t{strongly in  $L^2(Q_T)$  and a.e.\ in $Q_T$,}\label{eq:un-Leta}\\
T_k(u_n)&\to T_k(u) && \t{strongly in } L^2(0,T;H_0^1(\Omega)),\label{eq:conv-Tk}\\
\N u_n&\to \N u && \t{a.e. in } Q_T,\label{eq:Nun-ae}\\
T_n\pare{H(t,x,u_n,\N u_n)}&\to H(t,x,u,\N u) && \t{strongly in } L^1(Q_T),\label{eq:rhs-conv}\\
a(t,x,u_n, \N u_n)&\to  a(t, x, u_n, \N u) && \t{a.e.\ in $Q_T$ and weakly in $(L^2(Q_T))^N$,}\label{eq:Nun-w} \\
u_n&\to  u && \t{strongly in $C([0,T]; L^1(\Omega))$.}\label{convCL1}
\end{align}
\end{proposition}
\begin{proof}
\noindent
{\it The strong convergence of $u_n$ in $L^2(Q_T)$.} Since $\set{\av{\N u_n}}$ is uniformly bounded in $L^2(Q_T)$, we can apply \cite[Corollary 4]{Simon} with  $p=2$, $X=H_0^1(\Omega)$, $B=L^2(\Omega)$, $Y=W^{-1,s'}(\Omega)$ for some $s>N$.

\noindent
{\it The a.e.\ convergence of $u_n$.}
It follows directly, up to a new subsequence,  from the convergence in $L^2$.

\noindent
{\it The equi-integrability of the right hand side.}
It is enough to prove the equi-integrability of $\set{|\N u_n|^q|u_n|}$.\\
Let $E\subset Q_T$ be a measurable set. Having in mind Proposition \ref{prop:ape}, we may write
\begin{align*}
  \iint_E|\N u_n|^q|u_n|\, dx\, d\tau &=\iint_{E\cap\{|u_n|\le k\}}|\N u_n|^q|u_n|\, dx\, d\tau+\iint_{E\cap\{|u_n|> k\}}|\N u_n|^q|u_n|\, dx\, d\tau \\
 & \le k\iint_{E}|\N u_n|^q\, dx\, d\tau+\iint_{\{|u_n|> k\}}|\N u_n|^q|u_n|\, dx\, d\tau \\
 & \le k\left(\iint_{\Omega}|\N u_n|^2\, dx\, d\tau\right)^{\frac q 2}|E|^{1-\frac q2}+\frac1{k^{\sigma-1}}\iint_{Q_T}|\N u_n|^q|u_n|^\sigma\, dx\, d\tau\\
 & \le kc|E|^{1-\frac q2}+\frac c{k^{\sigma-1}}\,, \quad \forall k>0\,.
\end{align*}
Next choose $\eps>0$ and fix $k$ large enough to get $\frac c{k^{\sigma-1}}<\frac\eps2$. Finally find $\delta>0$ such that $|E|<\delta$ implies $kc|E|^{1-\frac q2}<\frac\eps2$. The equi-integrability of $\set{|\N u_n|^q|u_n|}$ follows.

\noindent
{\it The strong convergence of $T_k(u_n)$ in $L^2(0,T;H_0^1(\Omega)) $ for all $k>0$.} By the Dunford--Pettis Theorem, the previous step yields, up to subsequences,
the weak convergence of the right hand side in $L^1(Q_T)$. This fact allows us to apply \cite[Theorem 2.1]{P99}.

\noindent
{\it The a.e.\ convergence of $\N u_n$.}
The previous step implies the convergence of $\N u_n$ in measure, so that we deduce the pointwise convergence.

\noindent
{\it The a.e.\ convergence of $a(t,x,u_n,\N u_n)$ and $T_n(H(t,x,u_n, \N u_n))$.}
They are consequences of the a.e.\ convergence of $u_n$ and $\N u_n$.

\noindent
{\it The strong convergence of the right hand side in $L^1(Q_T)$.}
This follows from an application of the Vitali Theorem.

\noindent
{\it The weak convergence of $a(t,x,u_n, \N u_n)$ in $L^2(Q_T) $.}
It follows from the a.e convergence and the uniform bound of $\set{a(t,x,u_n, \N u_n)}$ in $L^2(Q_T) $, due to the uniform bound of $\big\{|\N u_n|^2\big\}$ and our assumptions.

\noindent{\it The strong convergence of $u_n$ in $C([0,T]; L^1(\Omega))$.} This can be proved in a similar way as in \cite[Proposition 6.4]{DGLS}.
\end{proof}

Now, for $t\in (0,T]$, let $v(\tau,x)$ be a function such that
$v\in L^2(0,t;H_0^1(\Omega))\cap L^\infty(Q_t)\cap C([0,t];L^2(\Omega))$, $\pa_\tau v\in L^2(0,t;H^{-1}(\Omega))$. Then, using the convergences proved in Proposition \ref{prop:cmpt}, it is easy to pass to the limit in each integral of the weak formulation \eqref{eq:sol-pbn}. Therefore $u$ is a solution in the sense of Definition \ref{defsol}.

\noindent{\it Further regularity of $u$.}
Using the estimates of Proposition \ref{prop:ape} and the convergences just proved in Proposition \ref{prop:cmpt}, we obtain that
\[
u\in L^\infty(0,T;L^\sigma(\Omega))\,,\qquad u^{\sigma/2}\in L^2(0,T;H_0^1(\Omega))\,,\qquad
\av{\N u}^q |u|^\sigma\in L^1(Q_T).\]
To conclude the proof, we want to show that $u\in C([0,T];L^\si(\Omega))$.
On account of $u\in L^\infty(0,T;L^\sigma(\Omega))$, we only need to show that the function
given by $\xi(t)=\int_\Omega |u(t,x)|^\sigma dx$ is continuous in $[0,T]$.\\
We begin by fixing $t\in [0,T]$ and taking $|T_k(u_n)
|^{\sigma-2}T_k(u_n)$ as a test function in \eqref{ftestu'} to get
\begin{multline}\label{eq:cont1}
\int_\Omega\Xi_k(u_n(t,x))\, dx-\int_\Omega\Xi_k(T_n(u^0(x)))\, dx
+(\sigma-1)\iint_{Q_t} |T_k(u_n)|^{\sigma-2} a(\tau,x,u_n,\N u_n)\cdot \nabla T_k(u_n)\,
dx\, d\tau\\
=\iint_{Q_t} T_n(H(\tau, x, u_n, \N u_n))|T_k(u_n)|^{\sigma-2}T_k(u_n)\, dx\, d\tau\,,
\end{multline}
where $\Xi_k(s)=\int_0^s |T_k(\eta)|^{\sigma-2}T_k(\eta)\, d\eta$.\\
It is not difficult to pass to the limit in $n$. Indeed, the first two terms pass to
the limit since $u_n(t,\cdot)\to u(t,\cdot)$ (by \eqref{convCL1}) and $T_n(u^0)\to u^0$ in $L^1(\Omega)$. Regarding the third term, it can be written as
\[\iint_{Q_t} |T_k(u_n)|^{\sigma-2} a(\tau,x,u_n,\N T_k(u_n))\cdot \nabla T_k(u_n)\,
dx\, d\tau\,.\]
Having in mind \eqref{eq:conv-Tk}, we deduce that both $\nabla T_k(u_n)$ and $a(
\tau,x,u_n,\N T_k(u_n))$ strongly converge in $(L^2(Q_T))^N$  by \eqref{A2}. Since the remaining factor is bounded,
we may let $n$ go to infinity. The limit on right hand side  is a consequence of
\eqref{eq:rhs-conv} and Lebesgue's theorem. Hence,
\eqref{eq:cont1} becomes
\begin{multline}\label{eq:cont2}
\int_\Omega\Xi_k(u(t,x))\, dx
+(\sigma-1)\iint_{Q_t} |T_k(u)|^{\sigma-2} a(\tau,x,u,\N u)\cdot \N u\,
dx\, d\tau\\
=\int_\Omega\Xi_k(u^0(x))\, dx+\iint_{Q_t} H(\tau, x, u, \N u)|T_k(u)|^{\sigma-2}T_k(u)
\, dx\, d\tau\,.
\end{multline}
Observe that our assumptions on $u^0$ and $f$ along with $\av{\N u}^q |u|^\sigma\in L^1(
Q_T)$ allow us to pass to the limit as $k$ tends to infinity in the right hand side (by
Lebesgue's theorem), and in the left hand side (thanks to \eqref{A2} and since   $u\in L^\infty(0,T;L^\sigma(\Omega))$ and $|u|^{\sigma/2}\in L^2(0,T;H_0^1(\Omega))$).
Therefore, we may let $k$ go to infinity in \eqref{eq:cont2}, obtaining
\begin{multline*}
\frac1\sigma\int_\Omega |u(t,x)|^\sigma\, dx\\
=\frac1\sigma\int_\Omega |u^0(x)|^\sigma\, dx-(\sigma-1)\iint_{Q_t} |u|^{\sigma-2} a(
\tau,x,u,\N u)\cdot \N u\, dx\, d\tau
+\iint_{Q_t} H(\tau, x, u, \N u)|u|^{\sigma-2}u\, dx\, d\tau\,.
\end{multline*}
Using function $\xi$, we have proven that
\[\xi(t)=\xi(0)-\sigma(\sigma-1)\iint_{Q_t} |u|^{\sigma-2} a(\tau,x,u,\N u)\cdot \N u\,
dx\, d\tau
+\sigma\iint_{Q_t} H(\tau, x, u, \N u)|u|^{\sigma-2}u\, dx\, d\tau\,,\]
and so the continuity of the right hand side implies that of $\xi$.

We finally point out that we may take $t=T$ and $T_k(u_n)$ as test function instead of $\av{T_k(u_n)}^{\si-2}T_k(u_n)$, and follow
the same argument to arrive at the energy identity \eqref{energy}.

\subsection{The case $\frac{2}{N+1}<q<\frac{4}{N+2}$: compactness and end of the proof}
In this Section we assume that $ \frac2{N+1} < q <\frac4{N+2}$ and that $\si= \frac{Nq}{2-q} \in (1,2)$. Now, Proposition \ref{prop:ape2} implies  the following compactness result:

\begin{proposition}[Compactness of the approximating sequence]\label{prop:cmpt2}
The sequence of the approximate solutions \eqref{eq:pbn} satisfies, up to subsequences,
\begin{align}
u_n&\to u &&\t{strongly in } L^\eta(Q_T)\t{ for } \eta=\frac q2 (N+2) \in\Big(\frac{N+2}{N+1}, 2\Big),\label{eq:un-Leta2}\\
u_n&\to u && \t{a.e. in } Q_T,\label{eq:un-ae2}\\
T_k(u_n)&\to T_k(u) && \t{strongly in } L^2(0,T;H_0^1(\Omega)),\label{eq:conv-Tk2}\\
\N u_n&\to \N u && \t{a.e. in } Q_T,\label{eq:Nun-ae2}\\
a(t,x,u_n, \N u_n)&\to a(t, x, u_n, \N u)&& \t{a.e. in } Q_T,\nonumber\\%
T_n\pare{H(t,x,u_n,\N u_n)}&\to H(t,x,u,\N u) && \t{strongly in } L^1(Q_T),\label{eq:rhs-conv2}\\
u_n&\to  u && \t{strongly in $C([0,T]; L^1(\Omega))$.}\label{convCL12}
\end{align}
\end{proposition}
\begin{proof}
The proof is very similar to that of Proposition \ref{prop:cmpt}. We highlight the main differences.

\noindent
{\it The strong convergence of $u_n$ in $L^\eta(Q_T)$.}
We perform the following computations:
\begin{align*}
\iint_{Q_T}  \av{\N u_n}^\eta\,dx\,dt
&= \iint_{Q_T}  \pare{1+\av{u_n}}^{\frac\eta 2(\si-2)} \av{\N u_n}^\eta\pare{1+\av{u_n}}^{\frac\eta 2(2-\si)}\,dx\,dt \\
&\le \pare{\iint_{Q_T} \pare{1+\av{u_n}}^{\si-2}\av{\N u_n}^2\,dx\,dt}^\frac{\eta}{2}\pare{\iint_{Q_T}(1+ |u_n|)^{\frac{\eta}{2-\eta}(2-\si)}\,dx\,dt}^{\frac{2-\eta}{2}}.
\end{align*}
Recalling \eqref{cor-GN2}, we set $\eta$ such that
\[
\frac{\eta}{2-\eta}(2-\si)=\si\frac{N+2}{N}\q\Leftrightarrow\q \eta=\si\frac{N+2}{N+\si} =\frac q2 (N+2).
\]
We apply \cite[Corollary 4]{Simon} with  $p=\eta$, $X=W_0^{1,\eta}(\Omega)$, $B=L^\eta(\Omega)$, $Y=W^{-1,s'}(\Omega)$ for some $s>N$.

\noindent
{\it The equi-integrability of the right hand side.}
It is enough to prove the equi-integrability of $\set{|\N u_n|^q|u_n|}$.
Let $E\subset Q_T$ be a measurable set. Having in mind Proposition \ref{prop:ape2}  and the previous step (notice that $q<\eta$), we may write
\begin{align*}
  \iint_E|\N u_n|^q|u_n|\, dx\, dt &=\iint_{E\cap\{|u_n|\le k\}}|\N u_n|^q|u_n|\, dx\, dt+\iint_{E\cap\{|u_n|> k\}}|\N u_n|^q|u_n|\, dx\, dt \\
 & \le k\iint_{E}|\N u_n|^q\, dx\, dt+\iint_{\{|u_n|> k\}}|\N u_n|^q|u_n|\, dx\, dt \\
 & \le k\left(\iint_{\Omega}|\N u_n|^\eta\, dx\, dt\right)^{\frac q \eta}|E|^{1-\frac q\eta}+\frac1{k^{\sigma-1}}\iint_{\Omega}|\N u_n|^q|u_n|^\sigma\, dx\, dt\\
 & \le kc|E|^{1-\frac q\eta}+\frac c{k^{\sigma-1}}\,, \quad \forall k>0\,,
\end{align*}
and we can conclude as in the case $\si\ge2$.
\end{proof}

We now finish the proof of Theorem \ref{teo:small} in this range of $q$, by showing the entropy formulation.\\

\noindent
{\it Entropy formulation when $1<\si<2$}.
We want to prove that $u$ is an entropy solution, in the sense of Definition \ref{entrsol}.
To this aim, fix $k>0$, $t\in(0,T]$ and a function $v(\tau,x)$ such that
\[
v\in L^2(0,t;H_0^1(\Omega))\cap L^\infty(Q_t)\cap C([0,t];L^1(\Omega)),\q \pa_\tau v\in L^2(0,t;H^{-1}(\Omega))+L^1(Q_t).
\]
We need a suitable regularization with respect to  time of $T_k(u_n-v)$.\\
Following Landes and Mustonen \cite{L,LM}, for every $\nu\in\bN$, consider $\Phi_\nu\in C([0,t];H_0^1(\Omega))$ the solution to problem
\begin{equation*}%\label{Landes}
\left\{\begin{array}{ll}
\frac1\nu \partial_t\Phi_\nu+\Phi_\nu=T_k(u_n-v),\\[3mm]
\Phi_\nu(0,\cdot)\in H_0^1(\Omega),
\end{array}\right.
\end{equation*}
where $\Phi_\nu(0,\cdot)$ is a sequence such that $\|\Phi_\nu(0)\|_{L^\infty(\Omega)}\le k$ and converges to $T_k(T_n(u^0)-v(0,\cdot))$ strongly in $L^q(\Omega)$ for all $1\le q<\infty$. The main properties of function $\Phi_\nu$ are:
\begin{subequations}\label{eq:prop}
\begin{align}
&\|\Phi_\nu\|_\infty\le k && \hbox{ for all }\nu\in\bN,\label{eq:prop1}\\
&\Phi_\nu\to T_k(u_n-v) && \hbox{ strongly in }L^2(0,t;H_0^1(\Omega))\label{eq:prop2}.
\end{align}
\end{subequations}
Obviously, $\partial_\tau\Phi_\nu\in L^2(0,t;H_0^1(\Omega))$ and so this function can be taken as test function in \eqref{eq:sol-pbn}.
\begin{align*}
&\intO u_n(t,x)\Phi_\nu(t,x)\,dx-\intO T_n(u^0(x))\Phi_\nu(0,x)\,dx-\int_{0}^t\langle\pa_\tau \Phi_\nu, u_n\rangle\,d\tau
\\
&\q+\iint_{Q_t}a(\tau,x,u_n,\N u_n)\cdot\N \Phi_\nu\,dx\,d\tau
=\iint_{Q_t}T_n\pare{H(\tau,x,u_n,\N u_n)}\Phi_\nu\,dx\,d\tau.
\end{align*}
 Integrating by parts, we observe that
\[
\intO u_n(t,x)\Phi_\nu(t,x)\,dx-\intO T_n(u^0(x))\Phi_\nu(0,x)\,dx-\int_{0}^t\langle\partial_\tau \Phi_\nu, u_n\rangle\,d\tau=\int_{0}^t\langle\partial_\tau u_n, \Phi_\nu \rangle\,d\tau,
\]
wherewith we deduce
\begin{multline*}
\int_{0}^t\langle\partial_\tau (u_n-v), \Phi_\nu\rangle\,d\tau+\int_{0}^t\langle\partial_\tau v, \Phi_\nu\rangle\,d\tau
+\iint_{Q_t}a(\tau,x,u_n,\N u_n)\cdot\N \Phi_\nu\,dx\,d\tau\\
=\iint_{Q_t}T_n\pare{H(\tau,x,u_n,\N u_n)}\Phi_\nu\,dx\,d\tau\,.
\end{multline*}
Observe that we may let $\nu$ go to infinity in all these terms, due to \eqref{eq:prop1} and \eqref{eq:prop2}, obtaining
\begin{multline*}
\int_{0}^t\langle\partial_\tau (u_n-v), T_k(u_n-v)\rangle\,d\tau+\int_{0}^t\langle\partial_\tau v, T_k(u_n-v)\rangle\,d\tau\\
+\iint_{Q_t}a(\tau,x,u_n,\N u_n)\cdot\N T_k(u_n-v)\,dx\,d\tau
=\iint_{Q_t}T_n\pare{H(\tau,x,u_n,\N u_n)}T_k(u_n-v)\,dx\,d\tau\,.
\end{multline*}
We finally point out that the first term can be written as
\[
\int_{0}^t\langle\partial_\tau (u_n-v), T_k(u_n-v)\rangle\,d\tau
=\intO \Theta_k(u_n(t,x)-v(t,x))\, dx-\intO \Theta_k(T_n(u^0(x))-v(0,x))\, dx,
\]
with $\Theta_k(s)=\int_0^sT_k(\sigma)\, d\sigma$. Therefore, we have arrived at
\begin{multline*}%\label{entr}
\intO \Theta_k(u_n(t,x)-v(t,x))\, dx-\intO \Theta_k(T_n(u^0(x))-v(0,x))\, dx+\int_{0}^t\langle\partial_\tau v, T_k(u_n-v)\rangle\,d\tau\\
+\iint_{Q_t}a(\tau,x,u_n,\N u_n)\cdot\N T_k(u_n-v)\,dx\,d\tau
=\iint_{Q_t}T_n\pare{H(\tau,x,u_n,\N u_n)}T_k(u_n-v)\,dx\,d\tau\,.
\end{multline*}
We now wish to pass to the limit for $n\to\infty$.
We have that
\begin{align*}
\intO \Theta_k(u_n(t)-v(t))\,dx&\to \intO \Theta_k(u(t)-v(t))\,dx,\\
\intO \Theta_k(T_n(u^0)-v(0))\,dx&\to \intO \Theta_k(u^0-v(0))\,dx,
\end{align*}
because of the convergence in $C([0,T];L^1(\Omega))$.\\
The limit
\[
\int_0^t\langle\partial_\tau v,T_k(u_n-v)\rangle \,d\tau\to \int_0^t\langle\partial_\tau v,T_k(u-v)\rangle \,d\tau
\]
can be obtained observing that
\[
T_k(u_n-v)=T_k(T_{k+M}(u_n)-v),\qq M=\norm{v}_{L^\infty(Q_T)},
\]
and using the convergence in \eqref{eq:conv-Tk}.\\
We now deal with
\[
\iint_{Q_t}a(\tau,x,u_n,\N u_n)\cdot\N T_k(u_n-v)\,dx\,d\tau
\]
in the following way. We first rewrite it as
\begin{align*}
&\iint_{Q_t}a(\tau,x,u_n,\N u_n)\cdot\N T_k(u_n-v)\,dx\,d\tau\\
&=\iint_{Q_t}  a(\tau,x,T_{k+M}(u_n),\N T_{k+M}(u_n))\cdot\N T_k(T_{k+M}(u_n)-v)\,dx\,d\tau\\
&=\iint_{Q_t}a(\tau,x,T_{k+M}(u_n),\N T_{k+M}(u_n))\cdot \N T_{k+M}(u_n) \chi_{\set{\av{T_{k+M}(u_n)-v }\le k}}\,dx\,d\tau\\
&\q-
\iint_{Q_t}a(\tau,x,T_{k+M}(u_n),\N T_{k+M}(u_n))\cdot\N v \chi_{\set{\av{T_{k+M}(u_n)-v }\le k}}\,dx\,d\tau.
\end{align*}
 The regularity assumptions on $v$, \eqref{eq:A}, \eqref{eq:un-ae2} and \eqref{eq:conv-Tk2} imply
\begin{align*}
&\iint_{Q_t}a(\tau,x,T_{k+M}(u_n),\N T_{k+M}(u_n))\cdot\N v\, \chi_{\set{\av{T_{k+M}(u_n)-v }\le k}}\,dx\,d\tau\\
&\qq\to \iint_{Q_t} a(\tau,x,T_{k+M}(u),\N T_{k+M}(u))\cdot\N v\, \chi_{\set{\av{T_{k+M}(u)-v }\le k}}\,dx\,d\tau,\\
&\iint_{Q_t}a(\tau,x,T_{k+M}(u_n),\N T_{k+M}(u_n))\cdot \N T_{k+M}(u_n)\,\chi_{\set{\av{T_{k+M}(u_n)-v }\le k}}\,dx\,d\tau\\
&\qq\to \iint_{Q_t}a(\tau,x,T_{k+M}(u),\N T_{k+M}(u))\cdot \N T_{k+M}(u)\,\chi_{\set{\av{T_{k+M}(u)-v }\le k}}\,dx\,d\tau.
\end{align*}
Finally,
\begin{align*}
\iint_{Q_t}T_n\pare{H(\tau,x,u_n,\N u_n)} T_k(u_n-v)\,dx\,d\tau&\to\iint_{Q_t}H(\tau,x,u,\N u) T_k(u-v)\,dx\,d\tau,
\end{align*}
thanks to   \eqref{eq:rhs-conv2}.

This  concludes the proof.

\subsection{The cases $q= \frac{2}{N+1} $, $0<q< \frac{2}{N+1}$: compactness and end of the proof}

Before addressing the convergence in the case $\si=1$, we need the following technical result.

\begin{lemma}\label{lem:mar}
The bounds
\begin{equation*}%\label{eq:kB}
\|v\|_{L^\infty(0,T;L^1(\Omega))}\le B\qq\t{and}\qq \iint_{Q_T} |\N T_k(v)|^2\,dx\,d\tau\le kB\qq\forall k>0
\end{equation*}
imply that there exist constants $C_i$ (only depending on $B$, $N$, $T$ and $|\Omega|$) such that
\begin{align*}
|\{(t,x)\in Q_T\>:\> |v|>k\}|&\le C_1 k^{-\frac{N+2}N},\\
|\{(t,x)\in Q_T\>:\> |\N v|>h\}|&\le C_2 h^{-\frac{N+2}{N+1}},\\
|\{(t,x)\in Q_T\>:\> |\N v|^q|v|>\lambda\}|&\le C_3 \lambda^{-\frac{N+2}{q(N+1)+N}},
\end{align*}
for all $k,\,h,\,\lm>0$.  
Therefore, we have three estimates in the Marcinkiewicz spaces (see \cite{BBV}):
\[
 v\in M^{\frac{N+2}{N}}(Q_T),\qq\av{\N v}\in M^{\frac{N+2}{N+1}}(Q_T)\qq\hbox{and}\qq
|\N v|^q v\in M^{\frac{N+2}{q(N+1)+N}}(Q_T).
\]
\end{lemma}
\begin{proof}
The first two regularities follow from \cite[Theorem 3.6]{AMST} (see also \cite[Lemma B1]{M}).\\
To see the third one, observe that
\[\{(t,x)\in Q_T\>:\>|\N v|^q|v|>\lambda\}\subset \{(t,x)\in Q_T\>:\>|\N v|>h\} \cup \{|v|>\lambda/h^q\}\]
since $|\N v(t,x)|\le h$ and $|v(t,x)|\le \lambda/h^q$ imply $|\N v(t,x)|^q|v(t,x)|\le\lambda$. Thus,
\begin{align*}
| \{(t,x)\in Q_T\>:\>|\N v|^q|v|>\lambda\}|&\le |\{(t,x)\in Q_T\>:\>|\N v|>h\}| + |\{|v|>\lambda/h^q\}|\le \frac{C_2}{h^{\frac{N+2}{N+1}}}+\frac{C_1\, h^{q\frac{N+2}{N}}}{\lambda^{\frac{N+2}{N}}}\,.
\end{align*}
Minimizing in $h$, we get $h=c \lambda^{\frac{N+1}{q(N+1)+N}}$, from where the estimate follows.
\end{proof}

\begin{proposition}[The cases $q= \frac{2}{N+1} $ \& $0<q< \frac{2}{N+1}$ and $\si=1$]\label{prop:cmpt-om}
We have that, up to subsequences,
\begin{align*}
u_n&\to u &&\t{strongly in } L^{\frac{N+2}{N+1}}(Q_T)\t{ when }q= \frac{2}{N+1},\\%
u_n&\to u &&\t{strongly in } L^\eta(Q_T),\,\eta<\frac{N+2}{N+1},\,0<q< \frac{2}{N+1},\\
T_k(u_n)&\to T_k(u) && \t{strongly in } L^2(0,T;H_0^1(\Omega)) ,\\
u_n&\to u && \t{a.e. in } Q_T,\\
\N u_n&\to \N u && \t{a.e. in } Q_T,\\
a(t,x,u_n,\N u_n)&\to a(t,x,u,\N u) && \t{a.e. in } Q_T,\nonumber\\
T_n\pare{H(t,x,u_n,\N u_n)}&\to H(t,x,u,\N u) && \t{strongly in } L^1(Q_T),\\
u_n&\to  u && \t{strongly in $C([0,T]; L^1(\Omega))$.}
\end{align*}
\end{proposition}
\begin{proof}
The proof of this result is very similar to the ones of Propositions \ref{prop:cmpt}, \ref{prop:cmpt2}, so we briefly detail the convergences whose  proofs are different.\\
{\it The strong convergence of $u_n$ in $L^\eta(Q_T)$.}
Let us consider first the case  $q= \frac{2}{N+1} $. The argument is similar to that developed in Proposition \ref{prop:cmpt2}: we use the H\"older inequality to estimate as
\[\iint_{Q_T}|\N u_n|^\eta dx\, d\tau\le \left(\iint_{Q_T}\frac{|\N u_n|^2}{1+|u_n|}\, dx\, d\tau\right)^{\eta/2} \left(\iint_{Q_T}(1+|u_n|)^{\frac{\eta}{2-\eta}}dx\, d\tau\right)^{(2-\eta)/2}.
\]
Due to Corollary \ref{cor-GN3}, the right hand side is bounded if we choose
\[\frac{N+2}{N}=\frac{\eta}{2-\eta}\q\Leftrightarrow\q \eta=\frac{N+2}{N+1}.\]
We then apply \cite[Corollary 4]{Simon} with  $p=\eta$, $X=W_0^{1,\eta}(\Omega)$, $B=L^\eta(\Omega)$, $Y=W^{-1,s'}(\Omega)$ for some $s>N$.

Next, $0<q< \frac{2}{N+1}$ is assumed. We recall \eqref{eq:A} and \eqref{H} and we  take $T_k(u_n)$ as test function obtaining
\begin{multline*}
\int_{\Omega}\Theta_k(u_n(t))\, dx+\iint_{Q_t}|\N T_k(u_n)|^2\, dx\, d\tau\\
\le \int_\Omega\Theta_k(u^0)\, dx+\iint_{Q_T}|\N u_n|^q|u_n||T_k(u_n)|\, dx\, d\tau+\iint_{Q_T}|f|\, |T_k(u_n)|\, dx\, d\tau\\
\le k\left[\int_\Omega |u^0|\, dx+\iint_{Q_T}|\N u_n|^q|u_n|\, dx\, d\tau+\iint_{Q_T}|f|\, dx\, d\tau\right]
\le kB
\end{multline*}
since $\{|\N u_n|^q|u_n|\}$ is uniformly bounded in $L^1(Q_T)$ by Propositions \ref{prop:ape-om} and \ref{prop:ape-1}. It is straightforward that then
\[\iint_{Q_T}|\N T_k(u_n)|^2\, dx\, d\tau\le kB\,.\]
On the other hand, for every $t\in[0,T]$, we have
\[\frac1k \int_{\Omega}\Theta_k(u_n(t))\, dx\le B\qq\forall k>0,\]
and letting $k$ go to 0, we deduce
$\int_{\Omega}|u_n(t)|\, dx\le B$.
Therefore,
\[\max_{t\in[0,T]}\int_{\Omega}|u_n(t)|\, dx\le B\,.\]
Then, we apply Lemma \ref{lem:mar}, deducing that $\set{\av{\N u_n}}$ is uniformly bounded in $M^{\frac{N+2}{N+1}}(Q_T)$, hence, in $L^{\eta}(Q_T)$ for all $1\le \eta<\frac{N+2}{N+1}$.
The result follows from \cite[Corollary 4]{Simon} with  $p=\eta$, $X=W_0^{1,\eta}(\Omega)$, $B=L^\eta(\Omega)$, $Y=W^{-1,s'}(\Omega)$ for some $s>N$.\\
{\it The equi-integrability of the right hand side.}
It is enough to check the equi-integrability of $\set{\av{\N u_n}^q |u_n|}$.\\
Let $q=\frac{2}{N+1}$ and $E\subset Q_T$ be a measurable set. We now apply Proposition \ref{prop:ape-om} and the boundedness of $|\N u_n|^\eta$ (note that $q<\eta$) to get
\begin{align*}
  &\iint_E|\N u_n|^q|u_n|\, dx\, d\tau\\ &=\iint_{E\cap\{|u_n|\le k\}}|\N u_n|^q|u_n|\, dx\, d\tau+\iint_{E\cap\{|u_n|> k\}}|\N u_n|^q|u_n|\, dx\, d\tau \\
  &\le k\iint_{E}|\N u_n|^q\, dx\, d\tau+\iint_{\{|u_n|> k\}}|\N u_n|^q|u_n|\, dx\, d\tau \\
  &\le k\left(\iint_{Q_t}|\N u_n|^\eta\, dx\, d\tau\right)^{\frac q \eta}|E|^{1-\frac q\eta}+\frac1{\log(1+|k|)^{\frac{N+1}{2}}}\iint_{Q_t}|\N u_n|^q|u_n|\log(1+|u_n|)^{\frac{N+1}{2}}\, dx\, d\tau\\
&  \le kc|E|^{1-\frac q\eta}+\frac c{\log(1+|k|)^{\frac{N+1}{2}}}\,, \quad \forall k>0\,.
\end{align*}
We may argue as in Proposition \ref{prop:cmpt} to obtain the desired equi-integrability.

In the case $0<q<\frac{2}{N+1}$, the sequence $\set{\av{\N u_n}^q |u_n|}$ is uniformly bounded in $L^\gamma(Q_T)$, for some $1<\gamma<\frac{N+2}{q(N+1)+N}$, owing to Lemma \ref{lem:mar}. Thus, the right hand side is uniformly bounded in $L^\gamma(Q_T)$. We then deduce the equi-integrability of $\set{\av{\N u_n}^q |u_n|}$.\\

From now on, we can follow the same steps as in the previous Sections  to finish the proof.
\end{proof}

\section{Proof of Proposition \ref{prop:energy} - Existence of finite energy solutions with small $q$}\label{sec:energy}

We  sketch the proof of the a priori estimate, since the convergence can be done as in Subsection \ref{ssec:cmptL2}.

\begin{proof}[Proof of Proposition \ref{prop:energy}]

We use $u_n$ as test function in \eqref{eq:pbn} so that, recalling \eqref{eq:A} and \eqref{H},  we obtain
\begin{align*}
&\frac12 \intO u_n^2(t)\,dx+\iint_{Q_t} \av{\N u_n }^2\,dx\,d\tau \le\iint_{Q_t} \av{\N u_n}^q u_n^2\,dx\,d\tau+\iint_{Q_t} |f| |u_n|\,dx\,d\tau +\frac12 \intO (u^0)^2\,dx.
\end{align*}
To estimate the integral with  $|\nabla u_n|^q$, we apply H\"older's inequality with exponents $\Big(\frac{2}{q},\frac{2}{2-q}\Big)$, then we interpolate between $L^2(\Omega)$ and $L^{2^*}(\Omega)$,  and finally we use Sobolev's and Young's inequalities:
\begin{align*}
\intO \av{\N u_n}^q u_n^2\,dx&\le \bigg(\intO \av{\N u_n}^2 \,dx\bigg)^\frac{q}{2}
\bigg(\intO |u_n|^{\frac{4}{2-q}}\,dx\bigg)^\frac{2-q}{2}\le c\bigg(\intO \av{\N u_n}^2\,dx\bigg)^{\frac{q}{4}(N+2)}
\bigg(\intO u_n^2\,dx\bigg)^\frac{4-qN}{4}\\
&\le c\eps^{-\frac{4}{q(N+2)}}\bigg(\intO \av{\N u_n}^2\,dx\bigg)\bigg(\intO u_n^2\,dx\bigg)^\frac{2}{N+2}+
c\eps^{\frac{4}{4-(N+2)q}}\intO u_n^2\,dx,
\end{align*}
for some $\eps>0$ to be chosen.
Integrating from $0$ to $t$, it yields
\begin{align*}
\iint_{Q_t} \av{\N u_n}^q u_n^2\,dx \,d\tau
&\le c_1(\eps)  \sup_{\tau\in[0,t]}\bigg(\int_\Omega u_n^2(\tau)\,dx\bigg)^\frac{2}{N+2}\iint_{Q_t}  \av{\N u_n}^2\,dx \,d\tau\nonumber
\\
&\q +
c_2T\eps^{\frac{4}{4-(N+2)q}} \sup_{\tau\in[0,t]}\int_\Omega u_n^2(\tau)\,dx.%
\end{align*}
The estimate of the integral with the forcing term is the same as in Proposition \ref{prop:ape}.
We  get
\begin{align*}
\iint_{Q_t} |f| |u_n|\,dx\,d\tau&\le   \int_0^T \normaF{f(\tau)}{L^m(\Omega)}^r\,d\tau
    + c_3 \sup_{\tau\in[0,t]}\bigg(\int_\Omega u_n^2(\tau)\,dx\bigg)^{\frac{r'-2}{2}} \iint_{Q_t}|\N u_n|^2 \,dx\,d\tau.
\end{align*}
We gather the previous estimates:
\begin{multline*}
\frac12 \intO u_n^2(t)\,dx+
\bra{1-c_1(\eps) \sup_{\tau\in[0,t]}\bigg(\int_\Omega u_n^2(\tau)\,dx\bigg)^\frac{2}{N+2}-
c_3 \sup_{\tau\in[0,t]}\bigg(\int_\Omega u_n^2(\tau)\,dx\bigg)^{\frac{r'-2}{2}}}
\iint_{Q_t} \av{\N u_n }^2\,dx\,d\tau\\
\le
c_2T\eps^{\frac{4}{4-(N+2)q}} \sup_{\tau\in[0,t]}\int_\Omega u_n^2(\tau)\,dx
+ \int_0^T \normaF{f(\tau)}{L^m(\Omega)}^r\,d\tau+\frac12 \intO (u^0)^2\,dx.
\end{multline*}
We fix $\eps,\,\de>0$ such that
\[
c_2T\eps^{\frac{4}{4-(N+2)q}}\le \frac{1}{8},\q
 c_1(\eps) \de^\frac{2}{N+2}+
c_3 \delta^{\frac{r'-2}{2}}\le \frac{1}{2}.
\]
We set $\norm{u^0}_{L^2(\Omega)}$ and $\norm{f}_{L^r(0,T;L^m(\Omega))}$ such that
\[
 \int_0^T \normaF{f(\tau)}{L^m(\Omega)}^r\,d\tau+\frac12 \intO (u^0)^2\,dx\le \frac{\de}{8}.
\]
Hence,  the proof of the a priori estimate can be concluded as in Proposition \ref{prop:ape}.
\end{proof}

\section*{Acknowledgments}
M. Magliocca has been  supported  by Grant RYC2021-033698-I, funded by the MCIU/AEI/10.13039/ 501100011033; by the project  “Análisis Matemático Aplicado y Ecuaciones Diferenciales” Grant PID2022-141187NB-I00 funded by MICIU /AEI /10.13039/501100011033/FEDER, UE and acronym “AMAED”. This publication is part of the project PID2022-141187NB-I00 funded by MICIU/AEI /10.13039/\break501100011033 and also FEDER, UE.
M. Magliocca has also been  funded  by the project
 PID2022-140494NA-I00 funded by MCIN/AEI/10.13039/501100011033/ FEDER, UE.

S. Segura de Le\'on has partially been supported by Grant
PID2022-136589NB-I00 funded by MCIN/ AEI/10.13039/501100011033 and by Grant RED2022-134784-T also funded by MCIN/AEI/10.13039/ 501100011033.

\end{document}